\documentclass[reqno,11pt]{amsart}

\usepackage{a4wide}
\usepackage[T1]{fontenc}
\usepackage{ae,aecompl}
\usepackage[english]{babel}
\usepackage[latin1]{inputenc}
\usepackage{enumerate}
\usepackage{latexsym}
\usepackage{amsmath,amsthm,amsfonts,amssymb}
\usepackage{xr}
\usepackage{graphicx}
\usepackage{stmaryrd}
\usepackage{mathtools}
\usepackage{braket}
\usepackage{pgfplots}
\usepackage{textcomp}
\usepackage{bbm}

\newtheorem{satz}{Theorem}[section]
\newtheorem{cor}[satz]{Corollary}

\newtheorem{prop}[satz]{Proposition}
\newtheorem{defi}[satz]{Definition}
\newtheorem{bem}[satz]{Remark}
\newtheorem{bsp}[satz]{Example}
\newtheorem{ass}[satz]{Assumption}

\numberwithin{equation}{section}

\newcommand{\sgn}{\operatorname{sgn}}
\newcommand{\erf}{\operatorname{erf}}
\renewcommand{\Re}{\operatorname{Re}}
\renewcommand{\Im}{\operatorname{Im}}

\newcommand{\Arg}{\operatorname{Arg}}
\newcommand{\AC}{\operatorname{AC}}

\title[Schrödinger evolution of superoscillations and supershifts]{\bf A unified approach to Schrödinger evolution of superoscillations and supershifts}

\author[Yakir Aharonov]{Yakir Aharonov}
\address{(YA) Schmid College of Science and Technology, Chapman University, Orange 92866, CA, US}
 \email{aharonov@chapman.edu}

\author[Jussi Behrndt]{Jussi Behrndt}
\address{(JB) Institut für Angewandte Mathematik, Technische Universität Graz, Steyrergasse 30, 8010 Graz, Austria}
\email{behrndt@tugraz.at}

\author[Fabrizio Colombo]{Fabrizio Colombo}
\address{(FC) Politecnico di Milano, Dipartimento di Matematica, Via E. Bonardi, 9, 20133 Milano, Italy}
\email{fabrizio.colombo@polimi.it}

\author[Peter Schlosser]{Peter Schlosser}
\address{(PS) Institut für Angewandte Mathematik, Technische Universität Graz, Steyrergasse 30, 8010 Graz, Austria}
\email{schlosser@tugraz.at}

\begin{document}

\begin{abstract}
Superoscillating functions and supershifts appear naturally in weak measurements in physics. Their evolution as
initial conditions in the time dependent Schrödinger equation is an important and challenging problem in quantum mechanics and mathematical analysis.
The concept that encodes the persistence of superoscillations during the evolution is the (more general)
supershift property of the solution.
In this paper we give a unified approach to determine the supershift property for the solution of the time dependent Schrödinger equation.
The main advantage and novelty of our results is that they only require suitable estimates and regularity assumptions on the Green's function, but not its explicit form.
With this efficient general technique we are able to treat various  potentials.
\end{abstract}
\maketitle
\medskip


\maketitle

\section{Introduction}

Superoscillations are band limited functions $F$ that oscillate faster than their fastest Fourier component; they appear
in connection with weak measurements in quantum mechanics \cite{aav,abook,av,b5} and as initial conditions in the time dependent Schrödinger equation
\begin{equation}\label{Eq_Time_Dependent_Schroedinger}
\begin{split}
i\frac{\partial}{\partial t}\Psi(t,x)&=\Big(-\frac{\partial^2}{\partial x^2}+V(t,x)\Big)\Psi(t,x), \\
\Psi(0,x)&=F(x),
\end{split}
\end{equation}
or in optics, signal processing, and other fields of physics and engineering as, e.g., antenna theory \cite{MBANTENNA,TDFG}.
The theory of superoscillations and their applications have grown enormously in the last decades and without claiming completeness we mention the contributions
\cite{berry2},\cite{berry}-\cite{b4}
and
\cite{Talbot,kempf1,kempf2,leeferreira2,lindberg}.
The standard example of a sequence of superoscillating functions is
\begin{equation}\label{Eq_Fn}
F_n(x;k)=\sum_{l=0}^n C_l(n;k)e^{ik_l(n)x},
\end{equation}
where $k>1$ and the coefficients $C_l(n;k)$, $k_l(n)$, for $n\in\mathbb{N}_0$, $l=0,\dots,n$, are given by
\begin{equation*}
C_l(n;k)={n\choose l}\bigg(\frac{1+k}{2}\bigg)^{n-l}\bigg(\frac{1-k}{2}\bigg)^l\quad\text{and}\quad k_l(n)=1-\frac{2l}{n},
\end{equation*}
using the binomial coefficient ${n\choose l}$. If we fix $x\in\mathbb{R}$ and let $n$ go to infinity, we obtain the limit
\begin{equation}\label{Eq_Fn_convergence}
\lim_{n\rightarrow\infty}F_n(x;k)=e^{ikx}.
\end{equation}
Observe that the frequencies $k_l(n)$ in \eqref{Eq_Fn} are in modulus bounded by $1$, but the frequency $k$ in the limit function in \eqref{Eq_Fn_convergence}
can be arbitrary large -
this (somewhat unexpected) behaviour gives rise to the notion {\it superoscillations}.
Inspired by the above example it has been shown that also other coefficients lead to the same phenomenon and the theory of superoscillations was extended
to a larger class of functions in \cite{newmeth,acsst6,acsst5} and to functions of several variables in \cite{JFAA}.

The quantum mechanical evolution problem of superoscillations investigates the behaviour of the solution of the time dependent Schrödinger equation
\eqref{Eq_Time_Dependent_Schroedinger}
with superoscillatory initial data. It is of particular importance to understand if a frequency shift of the initial conditions at time $t=0$, as in \eqref{Eq_Fn_convergence}, survives the time evolution and leads to a similar shift for the solutions at later times $t>0$.
The first Schrödinger evolution problem of superoscillations that has been studied was the free particle in \cite{acsst3}, where the solution gives a superoscillatory function in several variables. Later on, in the analysis of evolution problems with nonconstant potentials like the quantum harmonic oscillator \cite{uno,harmonic,YGERFn}, the electric field \cite{ACSST17}, or the uniform magnetic field \cite{J2}, it turned out that the solution of the Schrödinger equation with superoscillatory initial datum does not formally belong to the class of superoscillatory functions, although a certain frequency shift still
appears. These observations stimulated the notion of {\it supershift}, see, e.g., \cite{ABCS20,KG,J2,genHYP,Talbot}.
Very roughly speaking, it is known that the supershift property of the initial datum is stable under the time evolution of the Schrödinger equation for the above mentioned explicit cases. The analysis was based on sophisticated tools involving spaces of holomorphic functions with growth conditions, infinite order differential operators, but required for each of the potentials the explicit form of the
Green's function. For more details on superoscillations and supershifts we also refer the reader to
\cite{ABCS19,acsst4,acsst1,acsst3,KG,Aoki2,AOKI,Jussi}  and the introductory papers \cite{QS1,QS3,QS2, Be19, kempfQS}.

A general approach to study the evolution of superoscillations and supershifts, which only relies on qualitative properties of the Green's function and avoids its explicit form, does
not exist so far. It is the main objective of the
present paper to fill this gap; in fact, we shall provide a unified method,
where we just assume regularity and growth conditions on the Green's function.
The starting point in this paper will be the (formal) representation
\begin{equation}\label{Eq_Wave_function_integral}
\Psi(t,x)=\int_\mathbb{R}G(t,x,y)F(y)dy
\end{equation}
of the solution of the Schrödinger equation \eqref{Eq_Time_Dependent_Schroedinger} via the corresponding Green's function. The class of Green's functions which fit into our general setting is specified in Assumption~\ref{ass_Greensfunction}. In Section~\ref{sec_Fresnel_integrals} we develop the theory of Fresnel integrals, which will then be used in Section~\ref{GRENFREN} to give a rigorous meaning to the integral \eqref{Eq_Wave_function_integral}. More precisely, for a Green's function satisfying
Assumption~\ref{ass_Greensfunction} and an exponentially bounded initial condition
$F$ we show in Theorem~\ref{satz_Greensfunction} that \eqref{Eq_Wave_function_integral} can be viewed as
\begin{equation}\label{Eq_Psiintro}
\Psi(t,x)=\lim\limits_{\varepsilon\rightarrow 0^+}\int_\mathbb{R}e^{-\varepsilon y^2}G(t,x,y)F(y)dy=e^{i\alpha}\int_\mathbb{R}G(t,x,ye^{i\alpha})F(ye^{i\alpha})dy
\end{equation}
for some $\alpha\in(0,\frac{\pi}{2})$,
and in certain situations under slightly stronger assumption one even has
\begin{equation*}
\Psi(t,x)=\lim\limits_{R_1,R_2\rightarrow\infty}\int_{-R_1}^{R_2}G(t,x,y)F(y)dy;
\end{equation*}
cf. Remark~\ref{bembem}.
Moreover, we prove in Theorem~\ref{satz_Convergence_wavefunction} that the solution $\Psi$ depends continuously on the initial datum $F$. This result will be one of the main reasons why the supershift property is stable for $t>0$. In Section~\ref{CONTSUPS} we first recall the supershift property in a slightly more general form
in Definition~\ref{defi_Supershift},
and discuss its connection to the concept of superoscillations.
Afterwards we prove in Theorem~\ref{satz_Supershift_property} the time persistence of the supershift property for potentials $V$, where the corresponding Green's function satisfies  Assumption~\ref{ass_Greensfunction}. In the final Section~\ref{EXAMP} we apply our main results
to the Schrödinger equation with explicitly given potentials. Here we consider the free particle as a warm up, and
the time dependent uniform electric field, the time dependent harmonic oscillator, and the Pöschl-Teller potential as more sophisticated examples.
In each case  we verify that the corresponding Green's functions satisfies Assumption~\ref{ass_Greensfunction}, and hence fits into our general setting.
Therefore, for each of these potentials we conclude the time persistence of the supershift property of the initial condition.

\section{Fresnel integrals}\label{sec_Fresnel_integrals}

In this section we provide some preliminary material on the so-called Fresnel integral, which will be used in our main results in the next sections. Roughly speaking,
the
main purpose is to make sense of  integrals of the form
\begin{equation}\label{Eq_Non_integrable}
\int_\mathbb{R}e^{iy^2}f(y)dy,
\end{equation}
in particular, in the case where the function $f$ itself is not integrable.

\begin{prop}\label{prop_Fresnel_integral}
Let $f:\Omega\rightarrow\mathbb{C}$ be holomorphic on an open set $\Omega\subseteq\mathbb{C}$, which contains the sector
\begin{equation}\label{Eq_Sector}
S^+_\alpha\coloneqq\Set{z\in\mathbb{C} | 0\leq\Arg(z)\leq\alpha}
\end{equation}
for some $\alpha\in(0,\frac{\pi}{2})$. Then the following assertions hold.
\begin{itemize}
\item [{\rm (i)}]
If $f$ satisfies the estimate
\begin{equation}\label{Eq_Fresnel_estimate}
|f(z)|\leq Ae^{B|z|},\qquad z\in S^+_\alpha,
\end{equation}
for some $A,B\geq 0$, then for every $y_0\in\mathbb{R}$
\begin{equation}\label{Eq_Fresnel_integral}
\lim\limits_{\varepsilon\rightarrow 0^+}\int_0^\infty e^{-\varepsilon(y-y_0)^2}e^{iy^2}f(y)dy=e^{i\alpha}\int_0^\infty e^{i(ye^{i\alpha})^2}f(ye^{i\alpha})dy,
\end{equation}
where both integrands are absolutely integrable.
\item [{\rm (ii)}] If $f$ satisfies the estimate
\begin{equation}\label{Eq_Fresnel_estimate_Im}
|f(z)|\leq Ae^{B\Im(z)},\qquad z\in S_\alpha^+,
\end{equation}
for some $A,B\geq 0$, then
\begin{equation}\label{Eq_Fresnel_integral_Im22}
\lim\limits_{R\rightarrow\infty}\int_0^Re^{iy^2}f(y)dy=e^{i\alpha}\int_0^\infty e^{i(ye^{i\alpha})^2}f(ye^{i\alpha})dy,
\end{equation}
where the integrand on the right hand side is absolutely integrable, and also the integrand
 on the left hand side is absolutely integrable for every $R>0$.
\end{itemize}
\end{prop}

\begin{proof}
(i) Since the proof for arbitrary $y_0\in\mathbb{R}$ follows the same steps, we will restrict ourselves to $y_0=0$. We use the abbreviation $k=\tan(\alpha)>0$. For $R>0$ we consider the integration path

\begin{minipage}{0.5\textwidth}
\begin{align*}
\gamma_1&:=\Set{y \ |\ 0\leq y\leq R}, \\
\gamma_2&:=\Set{R+iy \ |\ 0\leq y\leq kR}, \\
\gamma_3&:=\Set{ye^{i\alpha} \ |\ 0\leq y\leq R\sqrt{1+k^2}}.
\end{align*}
\end{minipage}
\hspace{1cm}
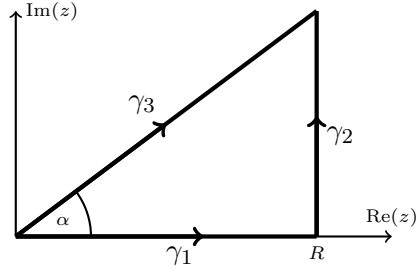
\begin{minipage}{0.39\textwidth}
\begin{tikzpicture}
\draw[thick,->] (0,0)--(5,0);
\draw[thick,->] (0,0)--(0,3);
\draw[ultra thick,->] (0,0)--(2.5,0);
\draw[ultra thick,-] (0,0)--(4,0);
\draw[ultra thick,->] (4,0)--(4,1.6);
\draw[ultra thick,-] (4,0)--(4,3);
\draw[ultra thick,->] (0,0)--(2,1.5);
\draw[ultra thick,-] (0,0)--(4,3);
\draw[thick,-] (1,0)arc(0:37:1);
\filldraw[black] (2.5,0) node[anchor=north east] {$\gamma_1$};
\filldraw[black] (4,1.6) node[anchor=north west] {$\gamma_2$};
\filldraw[black] (2,1.5) node[anchor=south east] {$\gamma_3$};
\filldraw[black] (5,0) node[anchor=south] {\tiny{$\Re(z)$}};
\filldraw[black] (0,3) node[anchor=west] {\tiny{$\Im(z)$}};
\filldraw[black] (0.4,0.2) node[anchor=west] {\tiny{$\alpha$}};
\filldraw[black] (4,0) node[anchor=north] {\tiny{$R$}};
\end{tikzpicture}
\end{minipage}
Since $f$ is holomorphic we have
\begin{equation}\label{Eq_Fresnel_integral1}
\int_{\gamma_1}e^{(i-\varepsilon)z^2}f(z)dz=-\int_{\gamma_2}e^{(i-\varepsilon)z^2}f(z)dz+\int_{\gamma_3}e^{(i-\varepsilon)z^2}f(z)dz.
\end{equation}
With the exponential bound \eqref{Eq_Fresnel_estimate}, the integral along the curve $\gamma_2$ can be estimated as
\begin{equation*}
\Big|\int_{\gamma_2}e^{(i-\varepsilon)z^2}f(z)dz\Big|\leq A\int_0^{kR}e^{\varepsilon y^2-2Ry-\varepsilon R^2}e^{B|R+iy|}dy\leq Ae^{-\varepsilon R^2+\sqrt{1+k^2}BR}\int_0^{kR}e^{\varepsilon y^2-2Ry}dy.
\end{equation*}
Choosing $\varepsilon\leq\frac{2}{k}$, the last integral can be estimated by
\begin{equation*}
\int_0^{kR}e^{\varepsilon y^2-2Ry}dy=kR\int_0^1e^{kR^2(\varepsilon ky-2)y}dy\leq kR,
\end{equation*}
and hence we conclude the convergence
\begin{equation*}
\lim\limits_{R\rightarrow\infty}\int_{\gamma_2}e^{(i-\varepsilon)z^2}f(z)dz=0.
\end{equation*}
Consequently, in the limit $R\rightarrow\infty$ \eqref{Eq_Fresnel_integral1} becomes
\begin{equation*}
\int_0^\infty e^{(i-\varepsilon)y^2}f(y)dy=e^{i\alpha}\int_0^\infty e^{(i-\varepsilon)(ye^{i\alpha})^2}f(ye^{i\alpha})dy,
\end{equation*}
where the integrand on the left hand side is absolutely integrable due to the factor $e^{-\varepsilon y^2}$, but also the right hand side is absolutely
integrable due to the estimates
\begin{align*}
\big|e^{(i-\varepsilon)(ye^{i\alpha})^2}f(ye^{i\alpha})\big|&\leq Ae^{-(\sin(2\alpha)+\varepsilon\cos(2\alpha))y^2+By}\\
&=Ae^{-\sin(2\alpha)\big(1+\varepsilon\frac{\cos^2(\alpha)-\sin^2(\alpha)}{2\sin(\alpha)\cos(\alpha)}\big)y^2+By} \\
&\leq Ae^{-\sin(2\alpha)(1-\frac{\varepsilon k}{2})y^2+By}\\
&\leq Ae^{-\sin(2\alpha)\frac{y^2}{2}+By}
\end{align*}
for all $\varepsilon<\frac{1}{k}$. Since this upper bound does not depend on $\varepsilon$ we can apply the dominated convergence theorem and obtain
\begin{equation*}
\lim\limits_{\varepsilon\rightarrow 0^+}\int_0^\infty e^{(i-\varepsilon)y^2}f(y)dy=e^{i\alpha}\int_0^\infty e^{i(ye^{i\alpha})^2}f(ye^{i\alpha})dy,
\end{equation*}
which is \eqref{Eq_Fresnel_integral}.

(ii)
In order to prove \eqref{Eq_Fresnel_integral_Im22} we use the same integration path as in (i) and conclude from \eqref{Eq_Fresnel_estimate_Im} for every $R>\frac{B}{2}$ the estimate
\begin{equation*}
\Big|\int_{\gamma_2}e^{iz^2}f(z)dz\Big|\leq\int_0^{kR}\Big|e^{i(R+iy)^2}f(R+iy)\Big|dy\leq A\int_0^\infty e^{-(2R-B)y}dy=\frac{A}{2R-B}.
\end{equation*}
Hence, we also conclude in this case
\begin{equation*}
\lim\limits_{R\rightarrow\infty}\int_{\gamma_2}e^{iz^2}f(z)dz=0.
\end{equation*}
Moroever, from \eqref{Eq_Fresnel_integral1} with $\varepsilon=0$ we obtain
\begin{equation*}
\lim\limits_{R\rightarrow\infty}\int_0^Re^{iy^2}f(y)dy=e^{i\alpha}\int_0^\infty e^{i(ye^{i\alpha})^2}f(ye^{i\alpha})dy,
\end{equation*}
where the integrand on the right hand side is absolutely integrable due to the estimate
\begin{equation*}
\big|e^{i(ye^{i\alpha})^2}f(ye^{i\alpha})\big|\leq Ae^{-\sin(2\alpha)y^2+B\sin(\alpha)y}.
\end{equation*}
\end{proof}

The next corollary is a slight extension of Proposition~\ref{prop_Fresnel_integral} which includes some additional constants $a>0$ and $y_1\in\mathbb{R}$.

\begin{cor}\label{cor_Fresnel_integral}
Let $a>0$, $y_1\in\mathbb{R}$ and let $f:\Omega\rightarrow\mathbb{C}$ be holomorphic on an open set $\Omega\subseteq\mathbb{C}$, which contains the sector $S_\alpha^+$ from \eqref{Eq_Sector} for some $\alpha\in(0,\frac{\pi}{2})$. Then the following assertions hold.

\begin{itemize}
\item [{\rm (i)}]
If $f$ satisfies the estimate
\begin{equation}\label{Eq_Fresnel_estimatecc}
|f(z)|\leq Ae^{B|z|},\qquad z\in S_\alpha^+,
\end{equation}
for some $A,B\geq 0$, then for every $y_0\in\mathbb{R}$
\begin{equation}\label{Eq_Fresnel_integralcc}
\lim\limits_{\varepsilon\rightarrow 0^+}\int_0^\infty e^{-\varepsilon(y-y_0)^2}e^{ia(y-y_1)^2}f(y)dy=e^{i\alpha}\int_0^\infty e^{ia(ye^{i\alpha}-y_1)^2}f(ye^{i\alpha})dy,
\end{equation}
where both integrands are absolutely integrable.
\item [{\rm (ii)}] If $f$ satisfies the estimate
\begin{equation}\label{Eq_Fresnel_estimate_Imcc}
|f(z)|\leq Ae^{B\Im(z)},\qquad z\in S_\alpha^+,
\end{equation}
for some $A,B\geq 0$, then
\begin{equation}\label{Eq_Fresnel_integral_Im22cc}
\lim\limits_{R\rightarrow\infty}\int_0^Re^{ia(y-y_1)^2}f(y)dy=e^{i\alpha}\int_0^\infty e^{ia(ye^{i\alpha}-y_1)^2}f(ye^{i\alpha})dy,
\end{equation}
where the integrand on the right hand side is absolutely integrable, and also the integrand on the left hand side is absolutely integrable for every $R>0$.
\end{itemize}
\end{cor}

\begin{proof}
(i) Substituting $x=y\sqrt{a}$, we can write
\begin{equation*}
\int_0^\infty e^{-\varepsilon(y-y_0)^2}e^{ia(y-y_1)^2}f(y)dy=\frac{1}{\sqrt{a}}\int_0^\infty e^{-\frac{\varepsilon}{a}(x-\sqrt{a}y_0)^2}e^{i(x-\sqrt{a}y_1)^2}f\Big(\frac{x}{\sqrt{a}}\Big)dx.
\end{equation*}
If we define $g(z)\coloneqq e^{i(-2z\sqrt{a}y_1+ay_1^2)}f(\frac{z}{\sqrt{a}})$, $z\in\Omega$, we can write this integral as
\begin{equation}\label{Eq_Fresnel_corollary1}
\int_0^\infty e^{-\varepsilon(y-y_0)^2}e^{ia(y-y_1)^2}f(y)dy=\frac{1}{\sqrt{a}}\int_0^\infty e^{-\frac{\varepsilon}{a}(x-\sqrt{a}y_0)^2}e^{ix^2}g(x)dx.
\end{equation}
Since, by \eqref{Eq_Fresnel_estimatecc}, $g$ satisfies the estimate
\begin{equation*}
|g(z)|=e^{2\Im(z)\sqrt{a}y_1}\Big|f\Big(\frac{z}{\sqrt{a}}\Big)\Big|\leq Ae^{\big(2\sqrt{a}|y_1|+\frac{B}{\sqrt{a}}\big)|z|},\qquad z\in S_\alpha^+,
\end{equation*}
we know by Proposition~\ref{prop_Fresnel_integral} (i), that
\begin{align*}
\lim\limits_{\varepsilon\rightarrow 0^+}\int_0^\infty e^{-\varepsilon(y-y_0)^2}e^{ia(y-y_1)^2}f(y)dy&=\frac{1}{\sqrt{a}}\lim\limits_{\varepsilon\rightarrow 0^+}\int_0^\infty e^{-\frac{\varepsilon}{a}(x-\sqrt{a}y_0)^2}e^{ix^2}g(x)dx \\
&=\frac{1}{\sqrt{a}}e^{i\alpha}\int_0^\infty e^{i(xe^{i\alpha})^2}g(xe^{i\alpha})dx \\
&=e^{i\alpha}\int_0^\infty e^{ia(ye^{i\alpha}-y_1)^2}f(ye^{i\alpha})dy.
\end{align*}

(ii) In the same way as in \eqref{Eq_Fresnel_corollary1} for $\varepsilon=0$, we can write
\begin{equation*}
\int_0^Re^{ia(y-y_1)^2}f(y)dy=\frac{1}{\sqrt{a}}\int_{0}^{R\sqrt{a}}e^{ix^2}g(x)dx.
\end{equation*}
Since, by \eqref{Eq_Fresnel_estimate_Imcc}, $g$ satisfies the estimate
\begin{equation*}
|g(z)|=e^{2\Im(z)\sqrt{a}y_1}\Big|f\Big(\frac{z}{\sqrt{a}}\Big)\Big|\leq Ae^{\bigl(2\sqrt{a}|y_1|+\frac{B}{\sqrt{a}}\bigr)\Im(z)},\qquad z\in S_\alpha^+,
\end{equation*}
we know by Proposition~\ref{prop_Fresnel_integral}~(ii), that
\begin{align*}
\lim\limits_{R\rightarrow\infty}\int_0^Re^{ia(y-y_1)^2}f(y)dy&=\frac{1}{\sqrt{a}}\lim\limits_{R\rightarrow\infty}\int_0^{R\sqrt{a}}e^{ix^2}g(x)dx \\
&=\frac{1}{\sqrt{a}}e^{i\alpha}\int_0^\infty e^{i(xe^{i\alpha})^2}g(xe^{i\alpha})dx \\
&=e^{i\alpha}\int_0^\infty e^{ia(ye^{i\alpha}-y_1)^2}f(ye^{i\alpha})dy.
\end{align*}
\end{proof}

The Fresnel integral technique in Proposition~\ref{prop_Fresnel_integral} and Corollary~\ref{cor_Fresnel_integral} can also be applied
on the negative semi axis. This leads to the following corollary.

\begin{cor}\label{cor_Fresnel_integral_R}
Let $a>0$, $y_1\in\mathbb{R}$ and let $f:\Omega\rightarrow\mathbb{C}$ be holomorphic on an open set $\Omega\subseteq\mathbb{C}$, which contains the double sector
\begin{equation}\label{Eq_Doublesector-}
S_\alpha\coloneqq\Set{z\in\mathbb{C} | \Arg(z)\in[0,\alpha]\cup[\pi,\pi+\alpha]}
\end{equation}
for some $\alpha\in(0,\frac{\pi}{2})$. Then the following assertions hold.
\begin{itemize}
\item [{\rm (i)}]
If $f$ satisfies the estimate
\begin{equation}\label{Eq_Fresnel_estimate-}
|f(z)|\leq Ae^{B|z|},\qquad z\in S_\alpha,
\end{equation}
for some $A,B\geq 0$, then for every $y_0\in\mathbb{R}$
\begin{equation}\label{Eq_Fresnel_integral_R-}
\lim\limits_{\varepsilon\rightarrow 0^+}\int_\mathbb{R}e^{-\varepsilon(y-y_0)^2}e^{ia(y-y_1)^2}f(y)dy=e^{i\alpha}\int_\mathbb{R}e^{ia(ye^{i\alpha}-y_1)^2}f(ye^{i\alpha})dy,
\end{equation}
where both integrands are absolutely integrable.
\item [{\rm (ii)}] If $f$ satisfies the estimate
\begin{equation}\label{Eq_Fresnel_estimate_Im-}
|f(z)|\leq Ae^{B|\Im(z)|},\qquad z\in S_\alpha,
\end{equation}
for some $A,B\geq 0$, then
\begin{equation}\label{Eq_Fresnel_integral_Im-}
\lim\limits_{R_1,R_2\rightarrow\infty}\int_{-R_1}^{R_2}e^{ia(y-y_1)^2}f(y)dy=e^{i\alpha}\int_\mathbb{R}e^{ia(ye^{i\alpha}-y_1)^2}f(ye^{i\alpha})dy,
\end{equation}
where the integrand on the right hand side is absolutely integrable, and also the integrand on the left hand side is absolutely integrable for every $R_1,R_2>0$.
\end{itemize}
\end{cor}

The initial purpose of the Fresnel integral technique was
to give meaning to the integral \eqref{Eq_Non_integrable}. For holomorphic functions $f$ satisfying the growth condition \eqref{Eq_Fresnel_estimate-},
Corollary~\ref{cor_Fresnel_integral_R}~(i) shows (with the choice $a=1$ and $y_1=0$) that one can
insert the Gaussian $e^{-\varepsilon(y-y_0)^2}$ and view the integral \eqref{Eq_Non_integrable} as the limit
\begin{equation*}
\lim\limits_{\varepsilon\rightarrow 0^+}\int_\mathbb{R}e^{-\varepsilon(y-y_0)^2}e^{iy^2}f(y)dy.
\end{equation*}
Under the stronger assumption \eqref{Eq_Fresnel_estimate_Im-} (which, in particular, implies that $f$ is bounded on the real line) such a
regularisation is not necessary and according to Corollary~\ref{cor_Fresnel_integral_R}~(ii) one can regard the integral \eqref{Eq_Non_integrable} as the limit
\begin{equation*}
\lim\limits_{R_1,R_2\rightarrow \infty}\int_{-R_1}^{R_2}e^{iy^2}f(y)dy.
\end{equation*}
However, under both assumptions \eqref{Eq_Fresnel_estimate-} and \eqref{Eq_Fresnel_estimate_Im-} one has the absolutely convergent representation
\begin{equation*}
e^{i\alpha}\int_\mathbb{R}e^{i(ye^{i\alpha})^2}f(ye^{i\alpha})dy.
\end{equation*}

\section{Green's functions and solutions of the Schr\"{o}dinger equation}\label{GRENFREN}

The main goal of this section is to treat the Cauchy problem \eqref{Eq_Time_Dependent_Schroedinger} for the time dependent Schrödinger equation
and the integral representation \eqref{Eq_Wave_function_integral} of the solution
in a mathematical rigorous framework. In particular, we provide a class of Green's functions and initial conditions
such that \eqref{Eq_Wave_function_integral} can be interpreted as a Fresnel integral.

\medskip

For our purposes it is convenient to view solutions (and their derivatives) in the context of absolutely continuous functions. Recall, that for an interval $I\subseteq\mathbb{R}$, a function $f:I\rightarrow\mathbb{C}$ is said to be \textit{absolutely continuous}, if there exists some $g\in L^1_\text{loc}(I)$, such that
\begin{equation}\label{Eq_Absolute_continuous_integral}
f(y)-f(x)=\int_x^y g(s)ds,\qquad x,y\in I.
\end{equation}
The linear space of absolutely continuous functions on $I$ will be denoted by $\AC(I)$. Observe that $f\in\AC(I)$ is differentiable almost everywhere and its derivative $f'$ coincides with $g$ in \eqref{Eq_Absolute_continuous_integral} almost everywhere. For $T\in(0,\infty]$ we shall work with the space
\begin{equation}\label{Eq_AC12}
\AC_{1,2}((0,T)\times\mathbb{R})\coloneqq\Set{\Psi:(0,T)\times\mathbb{R}\rightarrow\mathbb{C} | \begin{array}{l} \Psi(\,\cdot\,,x)\in\AC((0,T))\text{ for all }x\in\mathbb{R} \\ \Psi(t,\,\cdot\,),\frac{\partial}{\partial x}\Psi(t,\,\cdot\,)\in\AC(\mathbb{R})\text{ for all }t\in(0,T) \end{array}}.
\end{equation}

Let $V:(0,T)\times\mathbb{R}\rightarrow\mathbb{C}$ be some potential and let $F:\mathbb{R}\rightarrow\mathbb{C}$ be some initial condition. We call a function $\Psi\in\AC_{1,2}((0,T)\times\mathbb{R})$ a \textit{solution} of the time dependent Schrödinger equation, if it satisfies
\begin{subequations}\label{Eq_Schroedinger}
\begin{align}
i\frac{\partial}{\partial t}\Psi(t,x)&=\Big(-\frac{\partial^2}{\partial x^2}+V(t,x)\Big)\Psi(t,x), & \text{for a.e. }t\in(0,T),\,x\in\mathbb{R}, \label{Eq_Schroedinger1}\\
\lim\limits_{t\rightarrow 0^+}\Psi(t,x)&=F(x), & x\in\mathbb{R}. \label{Eq_Schroedinger2}
\end{align}
\end{subequations}
The corresponding \textit{Green's function} is a function $G:(0,T)\times\mathbb{R}\times\mathbb{R}\rightarrow\mathbb{C}$ (which depends on the potential $V$, but is independent of the initial condition $F$), such that $\Psi$ admits the representation
\begin{equation}\label{Eq_Psi_integral}
\Psi(t,x)=\int_\mathbb{R}G(t,x,y)F(y)dy,\qquad t\in(0,T),\,x\in\mathbb{R}.
\end{equation}

Next we collect a set of assumptions on the Green's function $G$ which ensure that the wave function \eqref{Eq_Psi_integral} is well defined and a solution of the Cauchy problem \eqref{Eq_Schroedinger}. The precise formulation of this statement, and also the set of allowed initial conditions, is given in Theorem~\ref{satz_Greensfunction}.

\begin{ass}\label{ass_Greensfunction}
Let $T\in(0,\infty]$ and consider a function
\begin{equation}\label{Eq_G_real}
G:(0,T)\times\mathbb{R}\times\mathbb R\rightarrow\mathbb{C}.
\end{equation}
Let $\Omega\subseteq\mathbb{C}$ be an open set, which contains the double sector $S_\alpha$ defined in  \eqref{Eq_Doublesector-} for some $\alpha\in(0,\frac{\pi}{2})$ and suppose that $G$ admits a continuation to a function $G:(0,T)\times\mathbb{R}\times\Omega\rightarrow\mathbb{C}$, such that $z\mapsto G(t,x,z)$ is holomorphic on $\Omega$ for every fixed $t\in(0,T)$, $x\in\mathbb{R}$. It will be assumed that the following properties {\rm (i)--(iv)} hold.

\begin{enumerate}
\item[{\rm (i)}] For every fixed $z\in S_\alpha$ the function $G(\,\cdot\,,\,\cdot\,,z)\in\AC_{1,2}((0,T)\times\mathbb{R})$
is a solution of the time dependent Schrödinger equation
\begin{equation}\label{Eq_Schroedinger_G}
i\frac{\partial}{\partial t}G(t,x,z)=\Big(-\frac{\partial^2}{\partial x^2}+V(t,x)\Big)G(t,x,z)\quad\text{for a.e. }t\in(0,T),\,x\in\mathbb{R}.
\end{equation}

\item[{\rm (ii)}] There exists $a\in\AC((0,T))$ with $a(t)>0$ and $\lim_{t\rightarrow 0^+}a(t)=\infty$, such that the function $\widetilde{G}$ in the decomposition
\begin{equation}\label{Eq_G_decomposition}
G(t,x,z)=e^{ia(t)(z-x)^2}\widetilde{G}(t,x,z),\qquad t\in(0,T),\,x\in\mathbb{R},\,z\in\Omega,
\end{equation}
satisfies
\begin{equation}\label{Eq_G_estimate}
|\widetilde{G}(t,x,z)|\leq A_0(t,x)e^{B_0(t,x)|z|},\qquad t\in(0,T),\,x\in\mathbb{R},\,z\in\Omega.
\end{equation}
Here $A_0,B_0:(0,T)\times\mathbb{R}\rightarrow[0,\infty)$ are nonnegative continuous functions, such that
\begin{equation}\label{Eq_Coefficient_limits}
\frac{A_0}{\sqrt{a}}\text{ and }B_0\text{ extend continuously to }[0,T)\times\mathbb{R}.
\end{equation}

\item[{\rm (iii)}] For all $x\in\mathbb{R}$ and  $z\in\Omega$ one has
\begin{equation}\label{Eq_G_initial_condition}
\lim\limits_{t\rightarrow 0^+}\frac{\widetilde{G}(t,x,z)}{\sqrt{a(t)}}=\frac{1}{\sqrt{i\pi}}.
\end{equation}

\item[{\rm (iv)}] There exist a nonnegative function $A_1\in L^1_\text{\rm loc}((0,T)\times\mathbb R)$ and a nonnegative continuous function $B_1:(0,T)\times\mathbb R\rightarrow [0,\infty)$, such that for every fixed $t\in(0,T)$ the spatial derivatives of $\widetilde G$ and for every fixed $x\in\mathbb{R}$ the time derivative of $\widetilde{G}$ are exponentially bounded, that is, the bounds
\begin{equation}\label{Eq_G_derivative_estimate}
\Big|\frac{\partial}{\partial x}\widetilde{G}(t,x,z)\Big|,\,\Big|\frac{\partial^2}{\partial x^2}\widetilde{G}(t,x,z)\Big|,\,\Big|\frac{\partial}{\partial t}\widetilde{G}(t,x,z)\Big|\leq A_1(t,x)e^{B_1(t,x)|z|}
\end{equation}
hold for all $z\in S_\alpha$.
\end{enumerate}
\end{ass}

We briefly comment on some of the conditions in Assumption~\ref{ass_Greensfunction} and also refer the reader to Section~\ref{EXAMP} for explicit examples of Green's
functions that satisfy Assumption~\ref{ass_Greensfunction}.

\begin{bem}
The assumption on the holomorphy of $G$ on $\Omega$ in the $z$-variable is needed to apply the Fresnel integral technique of Corollary~\ref{cor_Fresnel_integral_R}. In explicit examples and applications one typically starts with a Green's function as in \eqref{Eq_G_real} and verifies that it admits a holomorphic continuation to $G:(0,T)\times\mathbb{R}\times\Omega\rightarrow\mathbb{C}$.
The crucial assumption that allows to apply the Fresnel integral method is the decomposition \eqref{Eq_G_decomposition}, where the exponential $e^{ia(t)(z-x)^2}$
(quadratic in $z$) is separated from the remainder $\widetilde{G}$, which admits the (at most linear) exponential growth \eqref{Eq_G_estimate}. The rotation of the integration path in \eqref{Eq_Fresnel_integral_R-} from the real line into the complex plane, turns the factor $e^{ia(t)(z-x)^2}$ into a Gaussian which then dominates the linear exponential growth of $\widetilde{G}$ in the integral \eqref{Eq_Psi_integral}.
\end{bem}

\begin{bem}
We do not explicitly require the Green's function to converge to the delta distribution
\begin{equation}\label{Eq_G_delta}
\lim_{t\rightarrow 0^+}G(t,x,y)=\delta(x-y).
\end{equation}
In our situation, the counterpart of this standard assumption is the limit condition \eqref{Eq_G_initial_condition}, which is the key ingredient to ensure the initial value \eqref{Eq_Schroedinger2} of the wave function.
For us, the limit condition \eqref{Eq_G_initial_condition} is convenient, since in examples it is often easier to check than \eqref{Eq_G_delta}.
\end{bem}

The next theorem is the main result in the abstract part of this paper. It will be shown that under Assumption~\ref{ass_Greensfunction}
the integral \eqref{Eq_Psi_integral} is meaningful as a Fresnel integral and that the resulting function $\Psi$ is a solution of the time dependent Schrödinger equation \eqref{Eq_Schroedinger}.

\begin{satz}\label{satz_Greensfunction}
Let $G:(0,T)\times\mathbb{R}\times\mathbb{R}\rightarrow\mathbb{C}$ be as in Assumption~\ref{ass_Greensfunction}. Furthermore, let $F:\mathbb{R}\rightarrow\mathbb{C}$ be some initial condition, which admits a holomorphic continuation to the complex domain $\Omega$ from Assumption~\ref{ass_Greensfunction}, and satisfies the estimate
\begin{equation}\label{Eq_F_estimate}
|F(z)|\leq Ae^{B|z|},\qquad z\in\Omega,
\end{equation}
for some $A,B\geq 0$. Then the wave function
\begin{equation}\label{Eq_Psi}
\Psi(t,x)\coloneqq\lim\limits_{\varepsilon\rightarrow 0^+}\int_\mathbb{R}e^{-\varepsilon y^2}G(t,x,y)F(y)dy,\quad t\in(0,T),\,x\in\mathbb{R},
\end{equation}
exists and $\Psi\in \AC_{1,2}((0,T)\times\mathbb{R})$ is a solution of the Cauchy problem
\begin{subequations}
\begin{align}
i\frac{\partial}{\partial t}\Psi(t,x)&=\Big(-\frac{\partial^2}{\partial x^2}+V(t,x)\Big)\Psi(t,x),&\text{for a.e. }t\in(0,T),\,x\in\mathbb{R},\label{Eq_Schroedinger_Psi} \\
\lim\limits_{t\rightarrow 0^+}\Psi(t,x)&=F(x), & x\in\mathbb{R}. \label{Eq_Initial_Psi}
\end{align}
\end{subequations}
\end{satz}

\begin{bem}\label{bembem}
If we replace the growth condition \eqref{Eq_G_estimate} in Assumption~\ref{ass_Greensfunction}~(ii) by the stronger condition
\begin{equation}\label{Eq_G_estimate_Im}
|\widetilde{G}(t,x,z)|\leq A_0(t,x)e^{B_0(t,x)|\Im(z)|},\qquad t\in(0,T),\,x\in\mathbb{R},\,z\in\Omega,
\end{equation}
and also replace \eqref{Eq_F_estimate} by the stronger condition
\begin{equation}\label{asdd}
|F(z)|\leq Ae^{B|\Im(z)|},\qquad z\in\Omega,
\end{equation}
then it follows from Corollary~\ref{cor_Fresnel_integral_R}~(ii) that the wave function \eqref{Eq_Psi} can be written in the equivalent form
\begin{equation*}
\Psi(t,x)=\lim\limits_{R_1,R_2\rightarrow\infty}\int_{-R_1}^{R_2}G(t,x,y)F(y)dy,\quad t\in(0,T),\,x\in\mathbb{R}.
\end{equation*}
We point out that the stronger growth condition \eqref{Eq_G_estimate_Im} on  the Green's function is satisfied in all applications in Section~\ref{EXAMP}.
However, the growth condition \eqref{asdd} for the initial condition $F$ is rather restrictive and it is desirable to allow also initial conditions that may be unbounded on the real line. See for example the type of initial conditions which arise naturally for the supershift property in Theorem~\ref{satz_Supershift_property}.
\end{bem}

\begin{proof}[Proof of Theorem~\ref{satz_Greensfunction}]
\textit{Step 1.} In the first step we apply Corollary~\ref{cor_Fresnel_integral_R}, to show that the expression \eqref{Eq_Psi} for the wave function is meaningful. For this, we fix $t\in(0,T)$, $x\in\mathbb{R}$ and use the estimates \eqref{Eq_G_estimate} and \eqref{Eq_F_estimate} to get
\begin{equation}\label{Eq_Estimate_Integrand1}
|\widetilde{G}(t,x,z)F(z)|\leq AA_0(t,x)e^{(B+B_0(t,x))|z|},\qquad z\in S_\alpha.
\end{equation}
Hence, due to the decomposition \eqref{Eq_G_decomposition}, the assumptions of Corollary~\ref{cor_Fresnel_integral_R} are satisfied, which means, that the wave function \eqref{Eq_Psi} exists and admits the absolutely convergent representation
\begin{align}
\Psi(t,x)&=\lim\limits_{\varepsilon\rightarrow 0^+}\int_\mathbb{R}e^{-\varepsilon y^2}e^{ia(t)(y-x)^2}\widetilde{G}(t,x,y)F(y)dy \notag \\
&=e^{i\alpha}\int_\mathbb{R}e^{ia(t)(ye^{i\alpha}-x)^2}\widetilde{G}(t,x,ye^{i\alpha})F(ye^{i\alpha})dy \notag \\
&=e^{i\alpha}\int_\mathbb{R}G(t,x,ye^{i\alpha})F(ye^{i\alpha})dy. \label{Eq_Psi_absolute_convergent}
\end{align}

\noindent \textit{Step 2.} We show that the function $\Psi$ in \eqref{Eq_Psi}, is a solution of the Schrödinger equation \eqref{Eq_Schroedinger_Psi}. Roughly speaking, since $G$ is already a solution of \eqref{Eq_Schroedinger_G} by Assumption~\ref{ass_Greensfunction}~(i), one needs to check that wave function \eqref{Eq_Psi_absolute_convergent}
belongs to the space $\AC_{1,2}((0,T)\times\mathbb{R})$ and that the derivatives can be carried inside the integral.

Note first, that $G(\,\cdot\,,x,z)\in\AC((0,T))$ for every $x\in\mathbb{R}$, $z\in\Omega$ by Assumption~\ref{ass_Greensfunction}~(i) and hence for any $t_0\in(0,T)$ we have
\begin{equation*}
G(t,x,z)=G(t_0,x,z)+\int_{t_0}^t\frac{\partial}{\partial\tau}G(\tau,x,z)d\tau,\quad t\in(0,T),\,x\in\mathbb{R},\,z\in S_\alpha.
\end{equation*}
This leads to the following integral representation of the wave function \eqref{Eq_Psi_absolute_convergent}
\begin{equation}\label{Eq_Psi_derivative}
\Psi(t,x)=\Psi(t_0,x)+e^{i\alpha}\int_\mathbb{R}\int_{t_0}^t\frac{\partial}{\partial\tau}G(\tau,x,ye^{i\alpha})d\tau F(ye^{i\alpha})dy.
\end{equation}
Using the decomposition \eqref{Eq_G_decomposition} we can write the derivative as
\begin{equation*}
\frac{\partial}{\partial\tau}G(\tau,x,ye^{i\alpha})=\Big(ia'(\tau)(ye^{i\alpha}-x)^2\widetilde{G}(\tau,x,ye^{i\alpha})+\frac{\partial}{\partial\tau}\widetilde{G}(\tau,x,ye^{i\alpha})\Big)e^{ia(\tau)(ye^{i\alpha}-x)^2}.
\end{equation*}
Then we use \eqref{Eq_G_estimate}, \eqref{Eq_G_derivative_estimate} and \eqref{Eq_F_estimate} to estimate the integrand in \eqref{Eq_Psi_derivative} by
\begin{equation*}
\begin{split}
&\Big|\frac{\partial}{\partial\tau}G(\tau,x,ye^{i\alpha})F(ye^{i\alpha})\Big| \\
&\hskip 1cm \leq A\Big(A_0(\tau,x)|a'(\tau)||ye^{i\alpha}-x|^2e^{B_0(\tau,x)|y|}+A_1(\tau,x)e^{B_1(\tau,x)|y|}\Big) \\
&\hskip 8cm  \cdot e^{(B+2|x|a(\tau)\sin(\alpha))|y|}e^{-a(\tau)\sin(2\alpha)y^2}.
\end{split}
\end{equation*}
Since $A_0(\,\cdot\,,x)$, $B_0(\,\cdot\,,x)$, $B_1(\,\cdot\,,x)$, $a(\,\cdot\,)$ are continuous and $a'(\,\cdot\,),A_1(\,\cdot\,,x)\in L^1_\text{loc}((0,T))$
by Assumption~\ref{ass_Greensfunction}, it follows that the integrand in \eqref{Eq_Psi_derivative}
is integrable on $[t_0,t]$. Moreover, the factor $e^{-a(\tau)\sin(2\alpha)y^2}$ also implies integrability
with respect to $y\in\mathbb{R}$, and hence the integrand is absolutely integrable on $[t_0,t]\times\mathbb{R}$. Therefore, the order of integration in \eqref{Eq_Psi_derivative} can be interchanged and we obtain
\begin{equation*}
\Psi(t,x)=\Psi(t_0,x)+e^{i\alpha}\int_{t_0}^t\int_\mathbb{R}\frac{\partial}{\partial\tau}G(\tau,x,ye^{i\alpha})F(ye^{i\alpha})dyd\tau.
\end{equation*}
In particular, this shows $\Psi(\,\cdot\,,x)\in\AC((0,T))$ and the derivative with respect to $t$ exists almost everywhere and is given by
\begin{equation*}
\frac{\partial}{\partial t}\Psi(t,x)=e^{i\alpha}\int_\mathbb{R}\frac{\partial}{\partial t}G(t,x,ye^{i\alpha})F(ye^{i\alpha})dy.
\end{equation*}
Using the same argument, also $\Psi(t,\,\cdot\,)$ and $\frac{\partial}{\partial x}\Psi(t,\,\cdot\,)$ are absolutely continuous, with spatial derivatives almost everywhere given by
\begin{align*}
\frac{\partial}{\partial x}\Psi(t,x)&=e^{i\alpha}\int_\mathbb{R}\frac{\partial}{\partial x}G(t,x,ye^{i\alpha})F(ye^{i\alpha})dy, \\
\frac{\partial^2}{\partial x^2}\Psi(t,x)&=e^{i\alpha}\int_\mathbb{R}\frac{\partial^2}{\partial x^2}G(t,x,ye^{i\alpha})F(ye^{i\alpha})dy.
\end{align*}
This means $\Psi\in \AC_{1,2}((0,T)\times\mathbb{R})$
and from \eqref{Eq_Schroedinger_G} we conclude
that the Schrödinger equation \eqref{Eq_Schroedinger_Psi} is satisfied for a.e. $t\in (0,T)$, $x\in\mathbb R$.

\medskip

\noindent \textit{Step 3.} Now we verify the initial condition \eqref{Eq_Initial_Psi}. For this, we fix $x\in\mathbb R$ and split up the integral \eqref{Eq_Psi} as
\begin{equation*}
\Psi(t,x)=
\Psi_1(t,x)+\Psi_0(t,x)+\Psi_2(t,x),
\end{equation*}
where we have set
\begin{equation*}
\begin{split}
\Psi_1(t,x)&=\lim\limits_{\varepsilon\rightarrow 0^+}\int_{-\infty}^{y_1}e^{-\varepsilon y^2}G(t,x,y)F(y)dy, \\
\Psi_0(t,x)&=\lim\limits_{\varepsilon\rightarrow 0^+}\int_{y_1}^{y_2}e^{-\varepsilon y^2}G(t,x,y)F(y)dy, \\
\Psi_2(t,x)&=\lim\limits_{\varepsilon\rightarrow 0^+}\int_{y_2}^\infty e^{-\varepsilon y^2}G(t,x,y)F(y)dy,
\end{split}
\end{equation*}
and $y_1<0$ and $y_2>0$ are chosen such that $x\in(y_1,y_2)$. Starting with $\Psi_2$, we use the fact that the shifted sector $y_2+S_\alpha^+$  is contained in $S_\alpha^+$. A similar estimate as in \eqref{Eq_Estimate_Integrand1} shows, that we may apply Corollary~\ref{cor_Fresnel_integral}~(i)
to the shifted integrand $\widetilde{G}(t,x,y+y_2)F(y+y_2)$ and find
\begin{align*}
\Psi_2(t,x)&=\lim\limits_{\varepsilon\rightarrow 0^+}\int_{y_2}^\infty e^{-\varepsilon y^2}e^{ia(t)(y-x)^2}\widetilde{G}(t,x,y)F(y)dy \\
&=\lim\limits_{\varepsilon\rightarrow 0^+}\int_0^\infty e^{-\varepsilon(y+y_2)^2}e^{ia(t)(y+y_2-x)^2}\widetilde{G}(t,x,y+y_2)F(y+y_2)dy \\
&=e^{i\alpha}\int_0^\infty e^{ia(t)(ye^{i\alpha}+y_2-x)^2}\widetilde{G}(t,x,ye^{i\alpha}+y_2)F(ye^{i\alpha}+y_2)dy \\
&=e^{i\alpha}\int_0^\infty G(t,x,ye^{i\alpha}+y_2)F(ye^{i\alpha}+y_2)dy.
\end{align*}
In this form we can estimate $\Psi_2(t,x)$ as
\begin{align}
|\Psi_2(t,x)|&\leq\int_0^\infty\big|e^{ia(t)(ye^{i\alpha}+y_2-x)^2}\widetilde{G}(t,x,ye^{i\alpha}+y_2)F(ye^{i\alpha}+y_2)\big|dy \notag \\
&\leq AA_0(t,x)\int_0^\infty e^{-a(t)(y^2\sin(2\alpha)+2(y_2-x)y\sin(\alpha))}e^{(B+B_0(t,x))|ye^{i\alpha}+y_2|}dy \notag \\
&\leq AA_0(t,x)e^{(B+B_0(t,x))y_2}\int_0^\infty e^{-a(t)\sin(2\alpha)y^2}e^{(B+B_0(t,x)-2a(t)(y_2-x)\sin(\alpha))y}dy \notag \\
&=\frac{AA_0(t,x)\sqrt{\pi}}{2\sqrt{a(t)\sin(2\alpha)}}e^{(B+B_0(t,x))y_2}\Lambda\bigg(\frac{\sqrt{a(t)\tan(\alpha)}}{\sqrt{2}}\Big(y_2-x-\frac{B+B_0(t,x)}{2a(t)}\Big)\bigg), \label{Eq_Lambda_integral}
\end{align}
where, for a shorter notation, we used $\Lambda(\xi)\coloneqq e^{\xi^2}(1-\erf(\xi))$ as a modification of the well known error function; for the computation of the integral we refer to \cite[Lemma 2.1]{ABCS20}. According to Assumption~\ref{ass_Greensfunction}~(ii) we know that $\frac{A_0}{\sqrt{a}}$ and $B_0$ remain finite for $t\rightarrow 0^+$, and also that $a\rightarrow\infty$ for $t\rightarrow 0^+$. Therefore, since $\lim_{\xi\rightarrow\infty}\Lambda(\xi)=0$, see \cite[Formula 7.1.23]{AbSt72}, and $x<y_2$ we conclude
\begin{equation}\label{Eq_Psi2_limit}
\lim\limits_{t\rightarrow 0^+}\Psi_2(t,x)=0.
\end{equation}
In the same way one verifies, that
\begin{equation}\label{Eq_Psi1_limit}
\lim\limits_{t\rightarrow 0^+}\Psi_1(t,x)=0.
\end{equation}
For the limit of $\Psi_0$ for $t\rightarrow 0^+$ we first note that due to the dominated convergence theorem we can write
\begin{equation*}
\Psi_0(t,x)=\int_{y_1}^{y_2}G(t,x,y)F(y)dy.
\end{equation*}
Using the derivative $\frac{d}{d\xi}\erf(\xi)=\frac{2}{\sqrt{\pi}}e^{-\xi^2}$ of the error function together with the decomposition \eqref{Eq_G_decomposition} we rewrite the integral as
\begin{equation*}
\Psi_0(t,x)=\frac{\sqrt{\pi}}{2i\sqrt{ia(t)}}\int_{y_1}^{y_2}\bigg(\frac{\partial}{\partial y}\erf\big(i\sqrt{ia(t)}(y-x)\big)\bigg)\widetilde{G}(t,x,y)F(y)dy.
\end{equation*}
Applying integration by parts then leads to the four terms
\begin{equation}\label{Eq_Psi0}
\begin{split}
\Psi_0(t,x)=\frac{\sqrt{i\pi}}{2}\bigg(&-\erf\big(i\sqrt{ia(t)}(y_2-x)\big)\frac{\widetilde{G}(t,x,y_2)}{\sqrt{a(t)}}F(y_2) \\
&+\erf\big(i\sqrt{ia(t)}(y_1-x)\big)\frac{\widetilde{G}(t,x,y_1)}{\sqrt{a(t)}}F(y_1) \\
&+\int_{y_1}^{y_2}\erf\big(i\sqrt{ia(t)}(y-x)\big)\frac{\frac{\partial}{\partial y}\widetilde{G}(t,x,y)}{\sqrt{a(t)}}F(y)dy \\
&+\int_{y_1}^{y_2}\erf\big(i\sqrt{ia(t)}(y-x)\big)\frac{\widetilde{G}(t,x,y)}{\sqrt{a(t)}}F'(y)dy\bigg).
\end{split}
\end{equation}
Due to the limit
\begin{equation*}
\lim\limits_{a\rightarrow\infty}\erf\big(i\sqrt{ia}\,\xi\big)=\sgn(-\xi)=\left\{\begin{array}{ll} 1, & \text{if }\xi<0, \\ -1, & \text{if }\xi>0, \end{array}\right.
\end{equation*}
of the error function \cite[Formula 7.1.23]{AbSt72}, as well as the limit \eqref{Eq_G_initial_condition}, we find for the first two terms in \eqref{Eq_Psi0}
\begin{align*}
\lim\limits_{t\rightarrow 0^+}\erf\big(i\sqrt{ia(t)}(y_2-x)\big)\frac{\widetilde{G}(t,x,y_2)}{\sqrt{a(t)}}F(y_2)&=-\frac{F(y_2)}{\sqrt{i\pi}}, \\
\lim\limits_{t\rightarrow 0^+}\erf\big(i\sqrt{ia(t)}(y_1-x)\big)\frac{\widetilde{G}(t,x,y_1)}{\sqrt{a(t)}}F(y_1)&=\frac{F(y_1)}{\sqrt{i\pi}}.
\end{align*}
Consider now the fourth term in \eqref{Eq_Psi0}. Due to the estimate \eqref{Eq_G_estimate}, we have
\begin{equation*}
\frac{|\widetilde{G}(t,x,y)|}{\sqrt{a(t)}}\leq\frac{A_0(t,x)}{\sqrt{a(t)}}e^{B_0(t,x)|y|},\qquad y\in[y_1,y_2],
\end{equation*}
which is uniformly bounded for $t\rightarrow 0^+$ by \eqref{Eq_Coefficient_limits}. Since also the error function $\xi\mapsto\erf(i\sqrt{i}\,\xi)$ is bounded on $\mathbb{R}$, there exists a $t$-uniform majorant and hence the limit can be carried inside the integral. This gives
\begin{align*}
\lim\limits_{t\rightarrow 0^+}\int_{y_1}^{y_2}\erf\big(i\sqrt{ia(t)}(y-x)\big)\frac{\widetilde{G}(t,x,y)}{\sqrt{a(t)}}F'(y)dy&=\frac{1}{\sqrt{i\pi}}\int_{y_1}^{y_2}\sgn(x-y)F'(y)dy \\
&=\frac{2F(x)-F(y_1)-F(y_2)}{\sqrt{i\pi}}.
\end{align*}
For the third term in \eqref{Eq_Psi0} we note, that the compact interval $[y_1,y_2]$ is contained in the open set $\Omega$. Hence there exists some $r>0$ such that for every $y\in[y_1,y_2]$ the closed ball $B_r(y)$ is contained in $\Omega$. The Cauchy integral formula then gives
\begin{equation}\label{Eq_Gy_Cauchy_integral}
\frac{\partial}{\partial y}\widetilde{G}(t,x,y)=\frac{1}{2\pi i}\int_{\partial B_r(y)}\frac{\widetilde{G}(t,x,z)}{(z-y)^2}dz=\frac{1}{2\pi r}\int_0^{2\pi}\widetilde{G}(t,x,y+re^{i\varphi})e^{-i\varphi}d\varphi.
\end{equation}
Due to \eqref{Eq_G_estimate} we can estimate the integrand as
\begin{equation}\label{Eq_G_Cauchy_estimate}
\frac{|\widetilde{G}(t,x,y+re^{i\varphi})|}{\sqrt{a(t)}}\leq\frac{A_0(t,x)}{\sqrt{a(t)}}e^{B_0(t,x)|y+re^{i\varphi}|}\leq\frac{A_0(t,x)}{\sqrt{a(t)}}e^{B_0(t,x)(|y|+r)},
\end{equation}
and from \eqref{Eq_Coefficient_limits} we see that it admits an $t$-independent upper bound near $t=0^+$. Hence we are allowed to carry the limit inside the integral and get
\begin{equation}\label{Eq_Gy}
\lim\limits_{t\rightarrow 0^+}\frac{\frac{\partial}{\partial y}\widetilde{G}(t,x,y)}{\sqrt{a(t)}}=\frac{1}{2\pi r}\int_0^{2\pi}\lim\limits_{t\rightarrow 0^+}\frac{\widetilde{G}(t,x,y+re^{i\varphi})}{\sqrt{a(t)}}e^{-i\varphi}d\varphi=\frac{1}{2\pi r\sqrt{i\pi}}\int_0^{2\pi}e^{-i\varphi}d\varphi=0.
\end{equation}
Moreover, by the representation \eqref{Eq_Gy_Cauchy_integral} and the estimate \eqref{Eq_G_Cauchy_estimate} we get
\begin{equation*}
\frac{|\frac{\partial}{\partial y}\widetilde{G}(t,x,y)|}{\sqrt{a(t)}}\leq\frac{A_0(t,x)}{r\sqrt{a(t)}}e^{B_0(t,x)(|y|+r)},\qquad y\in[y_1,y_2].
\end{equation*}
Again, by \eqref{Eq_Coefficient_limits} we obtain an integrable and $t$-independent upper bound of the third term in \eqref{Eq_Psi0}. Hence we are allowed to carry the limit $t\rightarrow 0^+$ inside the integral, which vanishes according to \eqref{Eq_Gy}. Altogether this shows that \eqref{Eq_Psi0} converges to
\begin{equation*}
\lim\limits_{t\rightarrow 0^+}\Psi_0(t,x)=\frac{1}{2}\Big(F(y_2)+F(y_1)+2F(x)-F(y_1)-F(y_2)\Big)=F(x).
\end{equation*}
Together with \eqref{Eq_Psi2_limit} and \eqref{Eq_Psi1_limit} this finally proves the initial value \eqref{Eq_Initial_Psi}.
\end{proof}

In the next theorem we show that the solution $\Psi$ of the Schrödinger equation depends continuously on the initial condition. Note, that  the assumed convergence \eqref{Eq_A1_convergence} of the initial condition is stronger than the resulting uniform convergence on compact sets \eqref{Eq_Psi_convergence} at times $t>0$. However, the stronger convergence \eqref{Eq_A1_convergence} is justified, since this is the standard type of convergence in which superoscillations are normally treated, see for example \eqref{Eq_Superoscillations_convergence}. For convenience we use the notation $\Psi(t,x;F)$ to emphasize the initial condition.

\begin{satz}\label{satz_Convergence_wavefunction}
Let $G:(0,T)\times\mathbb{R}\times\mathbb{R}\rightarrow\mathbb{C}$ be as in Assumption~\ref{ass_Greensfunction}. Moreover, let $F,(F_n)_n:\mathbb{R}\rightarrow\mathbb{C}$ be initial conditions which admit holomorphic extensions to $\Omega$ and satisfy the growth condition \eqref{Eq_F_estimate} for some $A,B,(A_n)_n,(B_n)_n\geq 0$. If the sequence $(F_n)_n$ converges as
\begin{equation}\label{Eq_A1_convergence}
\lim\limits_{n\rightarrow\infty}\sup\limits_{z\in\Omega}|F(z)-F_n(z)|e^{-C|z|}=0
\end{equation}
for some $C\geq 0$, then also the corresponding wave functions converge as
\begin{equation}\label{Eq_Psi_convergence}
\lim\limits_{n\rightarrow\infty}\Psi(t,x;F_n)=\Psi(t,x;F)
\end{equation}
uniformly on compact subsets of $[0,T)\times\mathbb{R}$.
\end{satz}

\begin{proof}
Since the double sector $S_\alpha$ is contained in $\Omega$, there exists for every $x\in\mathbb{R}$ some $\beta(x)\in(0,\alpha]$, continuously depending on $x$, such that the shifted double sector $x+S_{\beta(x)}$ is contained in $\Omega$.
In the same way as in \eqref{Eq_Psi_absolute_convergent} we get the representation
\begin{equation}\label{Eq_Psi_beta}
\Psi(t,x;F)=e^{i\beta(x)}\int_\mathbb{R}G(t,x,x+ye^{i\beta(x)})F(x+ye^{i\beta(x)}),\qquad t\in(0,T),\,x\in\mathbb{R}.
\end{equation}
If we define
$$
L_n\coloneqq\sup_{z\in\Omega}|F(z)-F_n(z)|e^{-C|z|}
$$ from \eqref{Eq_A1_convergence}, we can use \eqref{Eq_G_decomposition} and \eqref{Eq_G_estimate} to estimate
\begin{align*}
|\Psi(t,x;F)-\Psi(t,x;F_n)|&\leq\int_\mathbb{R}\big|G(t,x,x+ye^{i\beta(x)})\big|\big|F(x+ye^{i\beta(x)})-F_n(x+ye^{i\beta(x)})\big|dy \\
&\leq L_nA_0(t,x)\int_\mathbb{R}e^{-a(t)\sin(2\beta(x))y^2}e^{(C+B_0(t,x))|x+ye^{i\beta(x)}|}dy \\
&\leq L_nA_0(t,x)e^{(C+B_0(t,x))|x|}\int_\mathbb{R}e^{-a(t)\sin(2\beta(x))y^2+(C+B_0(t,x))|y|}dy \\
&=\frac{L_nA_0(t,x)\sqrt{\pi}}{\sqrt{a(t)\sin(2\beta(x))}}e^{(C+B_0(t,x))|x|}\Lambda\bigg(-\frac{C+B_0(t,x)}{2\sqrt{a(t)\sin(2\beta(x))}}\bigg),
\end{align*}
where the analytic value of the last integral is similar to the one in \eqref{Eq_Lambda_integral}. Since $L_n\overset{n\rightarrow\infty}{\longrightarrow}0$ by \eqref{Eq_A1_convergence} and the right hand side is continuous in $t$ and $x$ by \eqref{Eq_Coefficient_limits} and the continuity of $\beta$, it follows that the convergence \eqref{Eq_Psi_convergence} is uniform on compact subsets of $[0,T)\times\mathbb{R}$.
\end{proof}

\section{Supershifts and superoscillations}\label{CONTSUPS}

The aim of this section is to investigate the time evolution of superoscillations and the supershift property of the solution $\Psi$ of the Schrödinger equation \eqref{Eq_Schroedinger}. As already mentioned in the introduction, the main novelty of our unified approach, with respect to the existing literature, is, that we are able to consider potentials, where the explicit form of the Green's function is not known. Instead, our results are only based on the regularity and growth conditions on the Green's functions, see Assumption~\ref{ass_Greensfunction}.

\medskip

We start with the abstract definition of a supershift and explain its usefulness and meaning afterwards.

\begin{defi}[Supershift]\label{defi_Supershift}
Let $\mathcal{O},\mathcal U\subseteq\mathbb{C}$ such that $\mathcal{U}\subsetneqq\mathcal{O}$. Let $X$ be a metric space and consider a family
\begin{equation}\label{phis}
\varphi_\kappa:X\rightarrow\mathbb C, \qquad \kappa\in\mathcal O,
\end{equation}
of complex valued functions. We say that a sequence of functions $(\Phi_n)_n$ of the form
\begin{equation}\label{Eq_Supershift_function}
\Phi_n(s)=\sum\limits_{l=0}^nC_l(n)\varphi_{\kappa_l(n)}(s)
,\qquad s\in X,
\end{equation}
with coefficients $C_l(n)\in\mathbb{C}$, $\kappa_l(n)\in\mathcal U$, admits a \emph{supershift}, if there exists some $\kappa\in\mathcal{O}\setminus\mathcal U$, such that
\begin{equation}\label{Eq_Supershift_convergence}
\lim\limits_{n\rightarrow\infty}\Phi_n(s)=\varphi_\kappa(s),\qquad s\in X,
\end{equation}
converges uniformly on compact subsets of $X$.
\end{defi}

\begin{bem}\label{analy}
If the sequence $(\Phi_n)_n$ in \eqref{Eq_Supershift_function} admits a supershift, then the values of $\varphi_\kappa$ for some $\kappa\in\mathcal O$, outside the smaller set $\mathcal{U}$, can be calculated by only using values $\varphi_{\kappa_l(n)}$ at the points
$\kappa_l(n)$ inside $\mathcal U$. Hence, informally speaking, when considering the mapping
$\kappa\mapsto\varphi_\kappa$ in the $\kappa$-variable there is a breeze of analyticity in the air, see also Theorem~\ref{satz_Analyticity}
and Corollary~\ref{cor_Exponentials}.
\end{bem}

Next we discuss a standard example for the supershift property, see also the example \eqref{Eq_Fn} in the introduction, as well as Remark~\ref{superbem} below for the connection to the notion of superoscillations.

\begin{bsp}\label{bsp_varphi_exponentials}
Let $X=\mathbb{C}$ and consider for every $\kappa\in\mathcal{O}=\mathbb{C}$ the exponentials
\begin{equation*}
\varphi_\kappa(z)=e^{i\kappa z},\qquad z\in\mathbb{C}.
\end{equation*}
Let furthermore $\mathcal U=[-1,1]$ and $\kappa\in\mathcal{O}\setminus \mathcal U$ be arbitrary. It was shown in \cite[Lemma 2.4]{YGERFn} that with the coefficients
\begin{equation}\label{Eq_Coefficients}
C_l(n)={n\choose l}\bigg(\frac{1+\kappa}{2}\bigg)^{n-l}\bigg(\frac{1-\kappa}{2}\bigg)^l
\qquad\text{and}\qquad\kappa_l(n)=1-\frac{2l}{n}\in\mathcal{U},
\end{equation}
the sequence of functions
\begin{equation*}
\Phi_n(z)=\sum_{l=0}^nC_l(n)e^{i\kappa_l(n)z},\qquad z\in\mathbb{C},
\end{equation*}
converges as
\begin{equation}\label{Eq_Fn_convergence2}
\lim_{n\rightarrow\infty}\Phi_n(z)=e^{i\kappa z}
\end{equation}
uniformly on compact subsets of $\mathbb{C}$, that is, for any $\kappa\in\mathbb{C}\setminus[-1,1]$, the sequence $(\Phi_n)_n$ admits a supershift. Moreover, according to \cite[Theorem 2.1]{YGERFn} one even has the stronger convergence
\begin{equation}\label{Eq_Superoscillations_convergence}
\lim\limits_{n\rightarrow\infty}\sup\limits_{z\in\mathbb{C}}|\Phi_n(z)-e^{i\kappa z}|e^{-C|z|}=0,
\end{equation}
for some $C\geq 0$, which agrees with the assumption \eqref{Eq_A1_convergence} in Theorem~{\rm\ref{satz_Convergence_wavefunction}}.
\end{bsp}

In the next remark we explain the connection between the notion of supershift in Definition~\ref{defi_Supershift} and the concept of superoscillations, which has attracted a lot of attention in the physical and mathematical literature, see the references mentioned in the introduction. Below we use the definition of superoscillations, for example, from \cite{ABCS20,acsst5}.

\begin{bem}[Superoscillations]\label{superbem}
Sequences of the form
\begin{equation}\label{basic_sequence}
\Phi_n(x)=\sum_{j=0}^n C_j(n)e^{ik_j(n)x},\qquad x\in\mathbb{R},
\end{equation}
with coefficients $C_j(n)\in\mathbb{C}$, $k_j(n)\in\mathbb{R}$, are often called {\em generalized Fourier sequences}. Note, that with
$\varphi_k(x)=e^{ikx}$, $x\in\mathbb{R}$, as in Example~\ref{bsp_varphi_exponentials}, this agrees with the functions in \eqref{Eq_Supershift_function}. The generalized Fourier sequence \eqref{basic_sequence} is
said to be {\em superoscillating} if there exists some $\tilde{k}\in\mathbb{R}$ such that
$$k':=\sup_{n\in\mathbb{N}_0,\,j\in\{0,\dots,n\}}|k_j(n)|<|\tilde{k}|,$$
and there exists a compact subset $K\subset\mathbb{R}$, called {\rm superoscillation set}, such that
\begin{equation}\label{Eq_Compact_convergence}
\lim\limits_{n\rightarrow\infty}\sup\limits_{x\in K}|\Phi_n(x)-e^{i\tilde{k}x}|=0;
\end{equation}
cf. \cite[Definition 5.1]{ABCS20}. Note, that the sequence $(\Phi_n)_n$ converges to a plane wave $e^{i\tilde{k}x}$ with frequency $|\tilde{k}|>k'$, that is, the shift in the $k$-variable is manifested in a shift of the frequencies, which leads to the terminology \textit{superoscillations}. Observe that superoscillations can be viewed as a special case of the supershift in Definition~\ref{defi_Supershift} by choosing
\begin{equation*}
\mathcal O=\mathbb R,\quad \mathcal U=[-k',k'],\quad\text{and}\quad X=K.
\end{equation*}
In fact, the supershift property is heavily inspired by the concept of superoscillations and Definition~\ref{defi_Supershift} is designed in such a way that
it applies to the time evolution of solutions of the Schr\"{o}dinger equation subject to superoscillatory initial data.
\end{bem}

The first main result of this section is the following Theorem~\ref{satz_Supershift_property} on the supershift property of the solution of the
Schr\"{o}dinger equation, which
can be viewed as a corollary of the continuous dependence result from Theorem~\ref{satz_Convergence_wavefunction}.
Roughly speaking we consider a family of initial conditions that admits a supershift (with respect to a slightly stronger form of convergence as in Definition~\ref{defi_Supershift}) and conclude that the corresponding solutions of the
Schrödinger equation admit a similar type of supershift; see also Remark~\ref{bem_Supershift}.

\begin{satz}[Supershift property]\label{satz_Supershift_property}
Let  $G:(0,T)\times\mathbb{R}\times\mathbb{R}\rightarrow\mathbb{C}$ and $\Omega$ be as in Assumption~\ref{ass_Greensfunction}, let
$\mathcal{O},\mathcal{U}\subseteq\mathbb{C}$ with $\mathcal{U}\subsetneqq\mathcal{O}$, and consider a family of analytic functions $\varphi_\kappa:\Omega\rightarrow\mathbb{C}$ with $\kappa\in\mathcal{O}$ that satisfy the estimate
\begin{equation}\label{Eq_varphi_exponential_bound}
|\varphi_\kappa(z)|\leq A(\kappa)e^{B(\kappa)|z|},\qquad z\in\Omega,
\end{equation}
for some $A(\kappa),B(\kappa)\geq 0$ continuously depending on $\kappa$. If a sequence of initial conditions $(F_n)_n$ of the form
\begin{equation}\label{Eq_Initial_Fn}
F_n(z)=\sum\limits_{l=0}^nC_l(n)\varphi_{\kappa_l(n)}(z),\qquad z\in\Omega,
\end{equation}
with coefficients $C_l(n)\in\mathbb{C}$, $\kappa_l(n)\in\mathcal U$, converge as
\begin{equation}\label{Eq_phi_kappa_convergence}
\lim\limits_{n\rightarrow\infty}\sup\limits_{z\in\Omega}|F_n(z)-\varphi_\kappa(z)|e^{-C|z|}=0,
\end{equation}
for some $C\geq 0$, to some $\varphi_\kappa$ with $\kappa\in\mathcal{O}\setminus\mathcal{U}$, then the sequence of solutions of the Schrödinger equation converges as
\begin{equation}\label{Eq_Psi_kappa_convergence}
\lim\limits_{n\rightarrow\infty}\Psi(t,x;F_n)=\lim\limits_{n\rightarrow\infty}\sum\limits_{l=0}^nC_l(n)\Psi(t,x;\varphi_{\kappa_l(n)})=\Psi(t,x;\varphi_\kappa)
\end{equation}
uniformly on compact subsets of $[0,T)\times\mathbb{R}$.
\end{satz}

\begin{proof}[Proof of Theorem~\ref{satz_Supershift_property}]
The fact, that the convergence \eqref{Eq_phi_kappa_convergence} leads to the convergence \eqref{Eq_Psi_kappa_convergence}, was already proven in Theorem~\ref{satz_Convergence_wavefunction}. Moreover, splitting the solutions $\Psi(t,x;F_n)$ into the given linear combination is allowed due to the linearity of the Schrödinger equation with respect to the initial condition, i.e.
\begin{equation*}
\Psi(t,x;F_n)=\Psi\Big(t,x;\sum\limits_{l=0}^nC_l(n)\varphi_{\kappa_l(n)}\Big)=\sum\limits_{l=0}^nC_l(n)\Psi(t,x;\varphi_{\kappa_l(n)}).
\end{equation*}
\end{proof}

\begin{bem}\label{bem_Supershift}
Since the convergence \eqref{Eq_phi_kappa_convergence} implies uniform convergence on all compact subsets of $\Omega$, it is clear
that the initial conditions $(F_n)_n$ in \eqref{Eq_Initial_Fn} admit the supershift property of Definition~\ref{defi_Supershift} with respect to the metric space $X=\Omega$. Furthermore, with the metric space $X=[0,T)\times\mathbb{R}$ and the functions \eqref{phis} as $\phi_\kappa(t,x):=\Psi(t,x;\varphi_\kappa)$, we are again in the setting of Definition~\ref{defi_Supershift}. The convergence \eqref{Eq_Psi_kappa_convergence} shows that the sequence $(\Psi(t,x;F_n))_n$ admits a supershift with respect to the functions $\phi_\kappa$ in the metric space $[0,T)\times\mathbb{R}$.
\end{bem}

In the next result we continue the theme of Theorem~\ref{satz_Supershift_property} and return to the analyticity issue mentioned in Remark~\ref{analy}. In fact, the following Theorem~\ref{satz_Analyticity} shows that analyticity in the $\kappa$-variable in the initial condition implies analyticity in the $\kappa$-variable
in the wave function.

\begin{satz}\label{satz_Analyticity}
Let  $G:(0,T)\times\mathbb{R}\times\mathbb{R}\rightarrow\mathbb{C}$ and $\Omega$ be as in Assumption~\ref{ass_Greensfunction},
let $\mathcal O\subseteq\mathbb C$ be an open set, and consider a family of analytic functions $\varphi_\kappa:\Omega\rightarrow\mathbb{C}$ with $\kappa\in\mathcal{O}$ that satisfy the estimate
\begin{equation}\label{Eq_varphi_estimate}
|\varphi_\kappa(z)|\leq A(\kappa)e^{B(\kappa)|z|},\qquad z\in\Omega,
\end{equation}
for some $A(\kappa),B(\kappa)\geq 0$ continuously depending on $\kappa$, and let $\Psi(t,x;\varphi_\kappa)$ be the corresponding
solution of the Schrödinger equation. If for every $z\in\Omega$ the mapping
\begin{equation*}
\mathcal{O}\ni\kappa\mapsto\varphi_\kappa(z)
\end{equation*}
is holomorphic, then for every fixed $t\in(0,T)$, $x\in\mathbb{R}$, the mapping
\begin{equation*}
\mathcal{O}\ni\kappa\mapsto\Psi(t,x;\varphi_\kappa)
\end{equation*}
is holomorphic as well.
\end{satz}

\begin{proof}
Fix $t\in(0,T)$, $x\in\mathbb{R}$. Then for any triangle $\Delta\subseteq\mathcal{O}$, we have the path integral
\begin{equation}\label{Eq_Psi_kappa2}
\int_\Delta\Psi(t,x;\varphi_\kappa)d\kappa=\int_\Delta e^{i\alpha}\int_\mathbb{R}G(t,x,ye^{i\alpha})\varphi_\kappa(ye^{i\alpha})dyd\kappa,
\end{equation}
due to the representation \eqref{Eq_Psi_absolute_convergent} of the wave function. Here $\alpha\in(0,\frac{\pi}{2})$ is the angle of the double sector $S_\alpha$ in Assumption~\ref{ass_Greensfunction}. In order to interchange the order of integration, we have to prove absolute integrability of the double integral. Firstly, the estimate
\begin{equation}\label{Eq_Psi_kappa1}
|G(t,x,ye^{i\alpha})\varphi_\kappa(ye^{i\alpha})|\leq A(\kappa)A_0(t,x)e^{-a(t)\sin(2\alpha)y^2}e^{(B(\kappa)+B_0(t,x)+2|x|a(t)\sin(\alpha))|y|}
\end{equation}
follows from \eqref{Eq_G_decomposition}, \eqref{Eq_G_estimate} and \eqref{Eq_varphi_estimate} and shows that the $y$-integral is absolutely convergent. Moreover, the coefficients $A(\kappa),B(\kappa)$ are assumed to be continuous and hence this upper bound can be uniformly estimated on the compact triangle $\Delta$. This means, that the right hand side of \eqref{Eq_Psi_kappa1} can be estimated by some $\kappa$-independent and $y$-integrable upper bound. Hence, the double integral \eqref{Eq_Psi_kappa2} is absolutely convergent and we are allowed to interchange the order of integration and get
\begin{equation*}
\int_\Delta\Psi(t,x;\varphi_\kappa)d\kappa=e^{i\alpha}\int_\mathbb{R}G(t,x,ye^{i\alpha})\int_\Delta\varphi_\kappa(ye^{i\alpha})d\kappa dy.
\end{equation*}
Since the mapping $\mathcal O\ni\kappa\mapsto\varphi_\kappa(ye^{i\alpha})$ is holomorphic the path integral along $\Delta$ vanishes and we get
\begin{equation*}
\int_\Delta\Psi(t,x;\varphi_\kappa)d\kappa=0.
\end{equation*}
Due to the Theorem of Morera this implies analyticity of $\mathcal{O}\ni\kappa\mapsto\Psi(t,x;\varphi_\kappa)$.
\end{proof}

In order to appreciate our main results, the following corollary shows how the above Theorem~\ref{satz_Supershift_property} and Theorem~\ref{satz_Analyticity} combine in the special case of the exponentials $\varphi_\kappa(z)=e^{i\kappa z}$ from Example~\ref{bsp_varphi_exponentials}, which are also the basic functions of superoscillations in Remark~\ref{superbem}.

\begin{cor}\label{cor_Exponentials}
Let $G:(0,T)\times\mathbb{R}\times\mathbb{R}\rightarrow\mathbb{C}$ be as in Assumption~\ref{ass_Greensfunction} and
consider the exponentials $\varphi_\kappa(z)=e^{i\kappa z}$, $\kappa,z\in\mathbb{C}$, as in Example~\ref{bsp_varphi_exponentials}. If for any $\kappa\in\mathbb{C}\setminus[-1,1]$ we choose the coefficients $C_l(n)$ and $\kappa_l(n)$ as in \eqref{Eq_Coefficients}, then the sequence of solutions $(\Psi(t,x;F_n))_n$ of the Schrödinger equation with initial condition
\begin{equation}\label{Eq_Initial_exponentials}
F_n(z)=\sum\limits_{k=0}^nC_l(n)e^{i\kappa_l(n)z},\qquad z\in\mathbb{C},
\end{equation}
converges as
\begin{equation}\label{Eq_Psi_kappa_convergencecor}
\lim\limits_{n\rightarrow\infty}\Psi(t,x;F_n)=\lim\limits_{n\rightarrow\infty}\sum\limits_{l=0}^nC_l(n)\Psi(t,x;e^{i\kappa_l(n)\,\cdot\,})=\Psi(t,x;e^{i\kappa\,\cdot\,})
\end{equation}
uniformly on compact subsets of $[0,T)\times\mathbb{R}$, that is, $(\Psi(t,x;F_n))_n$ admits a supershift. Moreover, for every fixed $t\in(0,T)$, $x\in\mathbb{R}$, the mapping
\begin{equation*}
\mathbb{C}\ni\kappa\mapsto\Psi(t,x;e^{i\kappa\,\cdot\,})
\end{equation*}
is analytic.
\end{cor}

\begin{proof}
According to \eqref{Eq_Superoscillations_convergence} the initial conditions $(F_n)_n$ converge as
\begin{equation*}
\lim\limits_{n\rightarrow\infty}\sup\limits_{z\in\mathbb{C}}|F_n(z)-e^{i\kappa z}|e^{-C|z|}=0,
\end{equation*}
for some $C\geq 0$. Hence it follows from Theorem~\ref{satz_Supershift_property}, that the sequence of solutions $(\Psi(t,x;F_n))_n$ converges as
\eqref{Eq_Psi_kappa_convergencecor}
uniformly on compact subsets of $[0,T)\times\mathbb{R}$.
Moreover, since the mapping $\mathbb{C}\ni\kappa\mapsto e^{i\kappa z}$ is analytic for every fixed $z\in\mathbb{C}$, it follows
from Theorem~\ref{satz_Analyticity} that also the mapping
\begin{equation*}
\mathbb{C}\ni\kappa\mapsto\Psi(t,x;e^{i\kappa\,\cdot\,})
\end{equation*}
is analytic for every fixed $t\in(0,T)$, $x\in\mathbb{R}$.
\end{proof}

\section{Some applications of the main results}\label{EXAMP}

We are now in the position to apply the results of the previous sections to the time dependent Schr\"{o}dinger equation with specific potentials. Here we consider the case of the free particle, the time dependent uniform electric field, the time dependent harmonic oscillator, and the Pöschl-Teller potential. Especially we formulate a variant of Theorem~\ref{satz_Supershift_property}. The proof of Theorem~\ref{aayy} is direct, in fact, for each potential we derive
the Green's function and verify the Assumption~\ref{ass_Greensfunction}.
From our discussion below it is also immediate that for each of the following potentials (I)--(IV), the representation of the wave function in Theorem~\ref{satz_Greensfunction}, the continuous dependency on the initial value in Theorem~\ref{satz_Convergence_wavefunction}, and the analyticity property in Theorem~\ref{satz_Analyticity} hold.

We also refer the reader to \cite{acsst3} for the free particle,
\cite{ACSST17} for the constant electric field and \cite{uno,harmonic,YGERFn} for the harmonic oscillator that are included in our general setting, while the  Pöschl-Teller potential is treated here for the first time. The difference with respect to the previous literature is that the initial condition we assume here is not necessarily a superoscillatory function but a supershift.

\begin{satz}\label{aayy}
Consider the time dependent Schr\"{o}dinger equation
\begin{subequations}\label{Eq_Schroedinger_app}
\begin{align}
i\frac{\partial}{\partial t}\Psi(t,x)&=\Big(-\frac{\partial^2}{\partial x^2}+V(t,x)\Big)\Psi(t,x), & \text{for a.e. }t\in(0,T),\,x\in\mathbb{R}, \label{Eq_Schroedinger_app1}\\
\lim\limits_{t\rightarrow 0^+}\Psi(t,x)&=F_n(x), & x\in\mathbb{R}, \label{Eq_Schroedinger_app2}
\end{align}
\end{subequations}
assume that the potential $V$ in \eqref{Eq_Schroedinger_app1} is one of the following:
\begin{enumerate}
\item[{\rm (I)}] $V(t,x)=0$,
\item[{\rm (II)}] $V(t,x)=\lambda(t)x$ with $\lambda:[0,\infty)\rightarrow\mathbb{R}$ continuous,
\item[{\rm (III)}] $V(t,x)=\lambda(t)x^2$ with $\lambda:[0,\infty)\rightarrow\mathbb{R}$ continuous,
\item[{\rm (IV)}] $V(t,x)=-\frac{l(l+1)}{\cosh^2(x)}$ for some $l\in\mathbb{N}$.
\end{enumerate}
Moreover, with $\mathcal{O},\mathcal{U}\subseteq\mathbb{C}$ such that $\mathcal{U}\subsetneqq\mathcal{O}$, and $\Omega\subseteq\mathbb{C}$ some open set containing the double sector $S_\alpha$ from \eqref{Eq_Doublesector-} for some $\alpha\in(0,\frac{\pi}{2})$, we consider a family of analytic functions $\varphi_\kappa:\Omega\rightarrow\mathbb{C}$, $\kappa\in\mathcal{O}$, and initial conditions of the form
\begin{equation}\label{Eq_Initial_Fn2}
F_n(z)=\sum\limits_{l=0}^nC_l(n)\varphi_{\kappa_l(n)}(z),\qquad z\in\Omega,
\end{equation}
where $C_l(n)\in\mathbb{C}$, $\kappa_l(n)\in\mathcal{U}$. If we assume that
\begin{equation}\label{Eq_phi_kappa_convergence_app}
\lim\limits_{n\rightarrow\infty}\sup\limits_{z\in\Omega}|F_n(z)-\varphi_\kappa(z)|e^{-C|z|}=0,
\end{equation}
for some $C\geq 0$ and $\varphi_\kappa$ with $\kappa\in\mathcal{O}\setminus\mathcal{U}$, then there exists $T\in(0,\infty]$ such that the sequence of solutions of \eqref{Eq_Schroedinger_app} converge as
\begin{equation}\label{Eq_Psi_convergence_app}
\lim\limits_{n\rightarrow\infty}\Psi(t,x;F_n)=\lim\limits_{n\rightarrow\infty}\sum\limits_{l=0}^nC_l(n)\Psi(t,x;\varphi_{\kappa_l(n)})=\Psi(t,x;\varphi_\kappa)
\end{equation}
uniformly on compact subsets of $[0,T)\times\mathbb{R}$.
\end{satz}

We point out that Theorem \ref{aayy} shows (similarly as Theorem \ref{satz_Supershift_property} and Remark \ref{bem_Supershift}) that for the potentials (I) -- (IV) the supershift property of the initial datum $(F_n)_n$, with the stronger convergence \eqref{Eq_phi_kappa_convergence_app}, carries over to a supershift property of the solutions $(\Psi(t,x;F_n))_n$. Indeed, the convergence \eqref{Eq_phi_kappa_convergence_app} implies the uniform convergence on all compact subsets of $\Omega$. Hence, $(F_n)_n$ admits a supershift according to Definition \ref{defi_Supershift} in the metric space $X=\Omega$ with respect to the functions $\varphi_\kappa$. Furthermore, the convergence \eqref{Eq_Psi_convergence_app} shows the supershift property of the wave functions $(\Psi(t,x;F_n))_n$ in the metric space $X=[0,T)\times\mathbb{R}$
with respect to the functions $\phi_\kappa(t,x):=\Psi(t,x;\varphi_\kappa)$.

\medskip

For the proof of Theorem~\ref{aayy}, the Green's functions for the respective potentials (I)--(IV) are investigated in the following paragraphs.
The first potential is the free particle, where it is almost obvious that Assumption~\ref{ass_Greensfunction} is satisfied.

\medskip

\subsection*{(I) Free particle $V(t,x)=0$} We show that in this case the Green's function
\begin{equation}\label{Eq_Gfree}
G(t,x,y)=\frac{1}{2\sqrt{i\pi t}}e^{-\frac{(y-x)^2}{4it}},\qquad t\in(0,\infty),\,x,y\in\mathbb{R},
\end{equation}
satisfies Assumption~\ref{ass_Greensfunction} with $T=\infty$. First of all it is clear, that for every $t\in(0,\infty)$ and $x\in\mathbb{R}$ we can extend
$G(t,x,\,\cdot\,)$ to an entire function by simply replacing $y\rightarrow z$. A direct computation shows that $G$ satisfies the differential equation \eqref{Eq_Schroedinger_G}. Moreover, the Green's function admits the decomposition \eqref{Eq_G_decomposition} with
\begin{equation*}
a(t,x)=\frac{1}{4t}\quad\text{and}\quad\widetilde{G}(t,x,z)=\frac{1}{2\sqrt{i\pi t}}.
\end{equation*}
The bound \eqref{Eq_G_estimate} and the extension properties \eqref{Eq_Coefficient_limits} are satisfied with the choices
\begin{equation*}
A_0(t,x)=\frac{1}{2\sqrt{\pi t}}\quad\text{and}\quad B_0(t,x)=0.
\end{equation*}
It is also clear that the limit condition \eqref{Eq_G_initial_condition} holds and the estimates for the derivatives in \eqref{Eq_G_derivative_estimate} follow immediately from the explicit form of $\widetilde{G}$.

\subsection*{(II) Time dependent uniform electric field $V(t,x)=\lambda(t)x$} It will be assumed that $\lambda:[0,\infty)\rightarrow\mathbb{R}$ is continuous. This type of
potential was already investigated with respect to the time persistence of superoscillations; cf. \cite[Theorem~3.6]{ACSST17}. In the present setting the Green's function is of the form
\begin{equation}\label{Eq_Greensfunction_Electric_field}
G(t,x,y)=\frac{1}{2\sqrt{i\pi t}}e^{i\beta(t)+ixt\alpha'(t)+iy\alpha(t)-\frac{(y-x)^2}{4it}},\qquad t>0,\,x,y\in\mathbb{R},
\end{equation}
where the coefficients $\alpha,\beta:[0,\infty)\rightarrow\mathbb{R}$ are the solutions of the ordinary differential equations
\begin{equation}\label{Eq_Coefficients_Electric_field}
t\alpha''(t)+2\alpha'(t)=-\lambda(t)\quad\text{and}\quad\beta'(t)=-t^2\alpha'(t)^2,\qquad t>0,
\end{equation}
with initial conditions $\alpha(0)=\beta(0)=\lim_{t\rightarrow 0^+}t\alpha'(t)=0$. It will be shown that $G$ satisfies
Assumption~\ref{ass_Greensfunction} with $T=\infty$. As above, it is clear that for $t\in(0,\infty)$ and $x\in\mathbb{R}$ one
can extend $G(t,x,\,\cdot\,)$ to an entire function by simply replacing $y\rightarrow z$. Differentiation and the use of the differential equations \eqref{Eq_Coefficients_Electric_field} show that $G$ is indeed a solution of \eqref{Eq_Schroedinger_G}.
Moreover, the Green's function admits the decomposition \eqref{Eq_G_decomposition} with
\begin{equation*}
a(t,x)=\frac{1}{4t}\quad\text{and}\quad\widetilde{G}(t,x,z)=\frac{1}{2\sqrt{i\pi t}}e^{i\beta(t)+ixt\alpha'(t)+iz\alpha(t)}.
\end{equation*}
It follows that the bound \eqref{Eq_G_estimate} and the extension properties \eqref{Eq_Coefficient_limits} are satisfied with the choices
\begin{equation*}
A_0(t,x)=\frac{1}{2\sqrt{\pi t}}\quad\text{and}\quad B_0(t,x)=|\alpha(t)|.
\end{equation*}
Using the initial conditions of the coefficients $\alpha$ and $\beta$ one easily verifies that the
limit condition \eqref{Eq_G_initial_condition} holds. Finally, when computing
the derivatives $\frac{\partial}{\partial x}\widetilde{G}$, $\frac{\partial^2}{\partial x^2}\widetilde{G}$, $\frac{\partial}{\partial t}\widetilde{G}$
one obtains
functions of the form
\begin{equation*}
P(t,x,z)e^{i\beta(t)+ixt\alpha'(t)+iz\alpha(t)},
\end{equation*}
where $P(t,x,z)$ is continuous in $t$ and $x$, and a polynomial in the $z$-variable. In this form it is not difficult to see that the derivatives can be estimated as in \eqref{Eq_G_derivative_estimate}.

\subsection*{(III) Time dependent harmonic oscillator $V(t,x)=\lambda(t)x^2$}  It will be assumed that $\lambda:[0,\infty)\rightarrow\mathbb{R}$ is continuous. In contrast to
the potentials (I) and (II), it turns out that the expression for the Green's function for the harmonic oscillator may only be valid on a finite time interval $(0,T)$. It is of the form
\begin{equation}\label{Eq_Greensfunction_Harmonic_oscillator}
G(t,x,y)=\frac{1}{2\sqrt{i\pi\alpha(t)}}e^{-\frac{\alpha'(t)x^2-2xy+\beta(t)y^2}{4i\alpha(t)}},\qquad t\in(0,T),\,x,y\in\mathbb{R},
\end{equation}
where the coefficients $\alpha,\beta$ are the solutions of the ordinary differential equations
\begin{equation}\label{Eq_Coefficients_Harmonic_oscillator}
\alpha''(t)=-4\lambda(t)\alpha(t)\qquad\text{and}\qquad\beta''(t)=-4\lambda(t)\beta(t),
\end{equation}
with initial conditions $\alpha(0)=\beta'(0)=0$ and $\beta(0)=\alpha'(0)=1$.
It will be shown that $G$ satisfies
Assumption~\ref{ass_Greensfunction} with $T>0$ chosen as the smallest positive zero of $\alpha$, or $T=\infty$ if $\alpha$ has no positive zero. With this choice of $T$ it follows from the initial conditions that $\alpha(t)>0$ for $t\in (0,T)$.
Again it is clear that for every $t\in(0,T)$ and $x\in\mathbb{R}$ one can extend $G(t,x,\,\cdot\,)$ to an entire function by simply replacing $y\rightarrow z$.
Note, that $\alpha$ and $\beta$ are linearly independent solutions of the
differential equation \eqref{Eq_Coefficients_Harmonic_oscillator} and hence the Wronskian has the constant value
\begin{equation*}
\alpha'(t)\beta(t)-\alpha(t)\beta'(t)=1,\qquad t>0.
\end{equation*}
Using this and \eqref{Eq_Coefficients_Harmonic_oscillator} one verifies by a straightforward computation that $G$ is a solution of
\eqref{Eq_Schroedinger_G}.
Moreover, the Green's function admits the decomposition \eqref{Eq_G_decomposition} with
\begin{equation*}
a(t,x)=\frac{\beta(t)}{4\alpha(t)}\quad\text{and}\quad\widetilde{G}(t,x,z)=\frac{1}{2\sqrt{i\pi\alpha(t)}}e^{\frac{(\beta(t)-\alpha'(t))x^2+2xz(1-\beta(t))}{4i\alpha(t)}}.
\end{equation*}
The bound \eqref{Eq_G_estimate} is satisfied with the choices
\begin{equation*}
A_0(t,x)=\frac{1}{2\sqrt{\pi\alpha(t)}}\quad\text{and}\quad B_0(t,x)=\frac{|x||1-\beta(t)|}{2\alpha(t)}.
\end{equation*}
Using $\beta(0)=1$ and
\begin{equation}\label{Eq_Limit}
\lim\limits_{t\rightarrow 0^+}\frac{1-\beta(t)}{\alpha(t)}=\lim\limits_{t\rightarrow 0^+}\frac{-\beta'(t)}{\alpha'(t)}=\frac{0}{1}=0,
\end{equation}
it follows that the extension properties \eqref{Eq_Coefficient_limits} hold. Using $\beta(0)=1$, the limit \eqref{Eq_Limit}, as well as
\begin{equation*}
\lim\limits_{t\rightarrow 0^+}\frac{\beta(t)-\alpha'(t)}{\alpha(t)}=\lim\limits_{t\rightarrow 0^+}\frac{\beta'(t)-\alpha''(t)}{\alpha'(t)}=\lim\limits_{t\rightarrow 0^+}\frac{\beta'(t)+4\lambda(t)\alpha(t)}{\alpha'(t)}=0,
\end{equation*}
one verifies that the limit condition \eqref{Eq_G_initial_condition} is satisfied. Finally, when computing
the derivatives $\frac{\partial}{\partial x}\widetilde{G}$, $\frac{\partial^2}{\partial x^2}\widetilde{G}$, $\frac{\partial}{\partial t}\widetilde{G}$ one
obtains functions of the form
\begin{equation*}
P(t,x,z)e^{\frac{(\beta(t)-\alpha'(t))x^2+2xz(1-\beta(t))}{4i\alpha(t)}},
\end{equation*}
where $P(t,x,z)$ is continuous in $t$ and $x$, and a polynomial in the $z$-variable.
In this form it is not difficult to see that the derivatives can be estimated as in \eqref{Eq_G_derivative_estimate}.

\begin{bem}
For the special case of a time-independent harmonic oscillator, this potential was already investigated with respect to the evolution of superoscillations in \cite[Proposition~5.2]{YGERFn}. The above considerations can be viewed as a time dependent generalization of the earlier results. Note, that in the particular situation $V(t,x)=\omega^2x^2$ with $\omega>0$ the Green's function \eqref{Eq_Greensfunction_Harmonic_oscillator} reduces to
\begin{equation*}
G(t,x,y)=\frac{\sqrt{\omega}}{\sqrt{2i\pi\sin(2\omega t)}}e^{-\frac{\omega(y-x)^2}{2i\tan(2\omega t)}-i\omega xy\tan(\omega t)},\qquad t\in\Big(0,\frac{\pi}{2\omega}\Big),\,x,y\in\mathbb{R},
\end{equation*}
and for $V(t,x)=-\omega^2x^2$ with $\omega>0$ the Green's function \eqref{Eq_Greensfunction_Harmonic_oscillator} becomes
\begin{equation*}
G(t,x,y)=\frac{\sqrt{\omega}}{\sqrt{2i\pi\sinh(2\omega t)}}e^{-\frac{\omega(y-x)^2}{2i\tanh(2\omega t)}+i\omega xy\tanh(\omega t)},\qquad t\in(0,\infty),\,x,y\in\mathbb{R}.
\end{equation*}
\end{bem}

\subsection*{(IV) Pöschl-Teller potential $V(t,x)=-\frac{l(l+1)}{\cosh^2(x)}$, $l\in\mathbb{N}$}
For the Pöschl-Teller potential it turns out that the Green's function cannot be extended to the whole complex plane, due to singularities on the imaginary axis.
For $x,y\in\mathbb{R}$ and $t>0$ the Green's function is given by
\begin{equation}\label{Eq_G_Poeschl_Teller}
G(t,x,y)=\bigg(\frac{1}{2\sqrt{i\pi t}}+\sum\limits_{m=1}^l\frac{m(l-m)!}{2(l+m)!}Q_l^m(x)Q_l^m(y)R\big(m^2t,m(y-x)\big)\bigg)e^{-\frac{(y-x)^2}{4it}}.
\end{equation}
where we use the function
\begin{equation}\label{Eq_R}
R(t,z)\coloneqq e^z\Lambda\Big(\frac{z}{2\sqrt{it}}-\sqrt{it}\Big)-e^{-z}\Lambda\Big(\frac{z}{2\sqrt{it}}+\sqrt{it}\Big),\qquad t>0,\,z\in\mathbb{C},
\end{equation}
with $\Lambda(z)\coloneqq e^{z^2}(1-\erf(z))$ a modification of the error function, $Q_l^m(x)\coloneqq P_l^m(\tanh(x))$ and $P_l^m$ the associated Legendre polynomials. This Green's function can for example be found in \cite[Section 6.6.3]{GS98}. It follows from the Legendre differential equation \cite[Eq. 8.1.1]{AbSt72} that $Q_l^m$ satisfies
\begin{equation}\label{Eq_Q_Differential_equation}
(Q_l^m)''(x)+\Big(\frac{l(l+1)}{\cosh^2(x)}-m^2\Big)Q_l^m(x)=0,\qquad x\in\mathbb{R}.
\end{equation}
Due to the representations \cite[Equations 8.6.6 \& 8.6.18]{AbSt72}, we know that the associated Legendre polynomials are of the form
\begin{equation*}
P_l^m(\xi)=(1-\xi^2)^{\frac{m}{2}}\Big(\text{polynomial in }\xi\Big),\qquad\xi\in(-1,1).
\end{equation*}
From this it follows that also $Q_l^m$ is of the form
\begin{equation}\label{Eq_Qlm_explicit}
Q_l^m(x)=\frac{1}{\cosh^m(x)}\Big(\text{polynomial in }\tanh(x)\Big),\qquad x\in\mathbb{R}.
\end{equation}
In particular, it is possible to extend $Q_l^m$ and hence also $G(t,x,\,\cdot\,)$ analytically to the complex domain $\mathbb{C}\setminus i\pi(\mathbb{Z}+\frac{1}{2})$, where only the zeros of the function $\cosh$ were excluded, that is, we consider $G(t,x,z)$ in \eqref{Eq_G_Poeschl_Teller} with $y$ replaced by $z\in\mathbb{C}\setminus i\pi(\mathbb{Z}+\frac{1}{2})$. Moreover, it can easily be checked, that
\begin{equation}\label{Eq_R_derivatives}
\begin{split}
\frac{\partial}{\partial z}R(t,z)&=\frac{z}{2it}R(t,z)-\frac{2}{\sqrt{i\pi t}}\sinh(z), \\
\frac{\partial}{\partial t}R(t,z)&=i\Big(1+\frac{z^2}{4t^2}\Big)R(t,z)+\frac{z\sinh(z)}{t\sqrt{i\pi t}}+\frac{2i\cosh(z)}{\sqrt{i\pi t}}.
\end{split}
\end{equation}
If one uses the derivatives \eqref{Eq_Q_Differential_equation} and \eqref{Eq_R_derivatives} together with the identity
\begin{equation*}
\sum\limits_{m=1}^l\frac{m(l-m)!}{(l+m)!}Q_l^m(z)\sinh\big(m(z-x)\big)Q_l^m(x)=\frac{l(l+1)}{4}\big(\tanh(z)-\tanh(x)\big)
\end{equation*}
for the Legendre polynomials, a straightforward (but long and technical) computation shows that \eqref{Eq_G_Poeschl_Teller} satisfies the Schrödinger equation \eqref{Eq_Schroedinger_G}.

\medskip

The Green's function admits the decomposition \eqref{Eq_G_decomposition} with $a(t)=\frac{1}{4t}$ and
\begin{equation*}
\widetilde{G}(t,x,z)=\frac{1}{2\sqrt{i\pi t}}+\sum\limits_{m=1}^l\frac{m(l-m)!}{2(l+m)!}Q_l^m(x)Q_l^m(z)R\big(m^2t,m(z-x)\big).
\end{equation*}
Now let us assume, that the domain $\Omega$ of Theorem~\ref{aayy} has a positive distance to the poles $i\pi(\mathbb{Z}+\frac{1}{2})$. If not, we can always shrink it, such that it has positive distance to $i\pi(\mathbb{Z}+\frac{1}{2})$ and still contains the double sector $S_\alpha$, e.g. by taking the intersection with
\begin{equation*}
\Omega_\alpha\coloneqq\Set{z\in\mathbb{C} | \vert\Im(z)\vert<\tan(\alpha)|\Re(z)|+\frac{\pi}{4}}.
\end{equation*}
Since the domain $\Omega$ now has positive distance to the zeros of $\cosh(z)$, we can use
\begin{equation*}
|\cosh(z)|^2=\sinh^2(\Re z)+\cos^2(\Im z)\geq c>0,\qquad z\in\Omega,
\end{equation*}
to estimate
\begin{equation}\label{Eq_Q_bound}
|Q_l^m(z)|\leq A_l^me^{B_l^m|z|},\qquad z\in\Omega,
\end{equation}
for some $A_l^m,B_l^m\geq 0$. Due to the properties \cite[Lemma 2.1]{ABCS20} of the function $\Lambda$ we can also estimate \eqref{Eq_R} as
\begin{equation}\label{Eq_R_bound}
|R(t,z)|\leq 2\Lambda\Big(-\frac{\sqrt{t}}{\sqrt{2}}\Big)e^{|\Re(z)|},\qquad t>0,\,z\in\mathbb{C};
\end{equation}
to verify this inequality it suffices to consider $\Re(z)+\Im(z)\geq 0$ due to the symmetry $R(t,-z)=R(t,z)$ and use the estimate as well as the monotonicity of $\Lambda$. Summing up we conclude that for $t>0$, $x\in\mathbb{R}$ and $z\in\Omega$ the function $\widetilde{G}$ can be estimated as
\begin{equation*}
|\widetilde{G}(t,x,z)|\leq\frac{1}{2\sqrt{\pi t}}+\sum\limits_{m=1}^l\frac{m(l-m)!}{(l+m)!}(A_l^m)^2\Lambda\Big(-\frac{m\sqrt{t}}{\sqrt{2}}\Big)e^{(m+B_l^m)(|x|+|z|)},
\end{equation*}
from which the bound \eqref{Eq_G_estimate} follows if we choose the coefficients
\begin{align*}
A_0(t,x)&=\frac{1}{2\sqrt{\pi t}}+\sum\limits_{m=1}^l\frac{m(l-m)!}{(l+m)!}(A_l^m)^2e^{(m+B_l^m)|x|}\Lambda\Big(-\frac{m\sqrt{t}}{\sqrt{2}}\Big), \\
B_0(t,x)&=\max\limits_{1\leq m\leq l}(m+B_l^m).
\end{align*}
Using $\Lambda(0)=1$ it is not difficult to see that these coefficients can be continuously extended as in \eqref{Eq_Coefficient_limits}. For the limit \eqref{Eq_G_initial_condition} we use the asymptotics
\begin{equation*}
R(t,z)=\frac{4\sinh(z)\sqrt{it}}{z\sqrt{\pi}}+\mathcal{O}(t),\quad\text{as }t\rightarrow 0^+,\quad\forall z\in\mathbb{C},
\end{equation*}
which follows from \cite[Lemma 2.1]{ABCS20}. Finally, using the derivative of $\Lambda$ from \cite[Lemma 2.1]{ABCS20} and the fact that $Q_l^m$ contains only
powers of the functions $\cosh$ and $\sinh$, as well as the inequalities \eqref{Eq_Q_bound} and \eqref{Eq_R_bound}, one can
show that also the derivatives of $\widetilde{G}$ are
exponentially bounded as required in \eqref{Eq_G_derivative_estimate}.

\end{document}